 \documentclass[12pt]{colt2019} 

\usepackage{algorithm}
\usepackage{algorithmic}
\usepackage{caption}

\usepackage{mathtools}
\mathtoolsset{showonlyrefs} 
\usepackage{dirtytalk}
\usepackage{tikz}
\newcommand*\circled[1]{\tikz[baseline=(char.base)]{
		\node[shape=circle,draw,inner sep=2pt] (char) {#1};}}

\newcommand*{\R}{{\mathbb R}}

\newcommand*{\e}{\varepsilon}
\newcommand*{\la}{\langle}
\newcommand*{\ra}{\rangle}

\newcommand*{\norm}[1]{\left\lVert#1\right\rVert}

\newcommand\uprule{\rule{0mm}{1.9ex}}
\DeclareMathOperator*{\argmin}{argmin}

\def\pd#1{{\color{black}#1}}
\def\ag#1{{\color{black}#1}}
\def\gav#1{{\color{black}#1}}
\def\at#1{{\color{black}#1}}
\def\dd#1{{\color{black}#1}} 
\def\at2#1{{\color{black}#1}}
\def\ak#1{{\color{black}#1}}

\newcommand{\leqarg}[1]{\ensuremath{\stackrel{\text{#1}}{\leq}}}
\newcommand{\eqarg}[1]{\ensuremath{\stackrel{\text{#1}}{=}}}

\allowdisplaybreaks

\title[Inexact Model in Optimization]{\pd{Inexact Model: A Framework for Optimization and Variational Inequalities}}
\usepackage{times}

\coltauthor{%
 \Name{Fedor Stonyakin} \Email{fedyor@mail.ru}\\
 \addr V. Vernadsky Crimean Federal University, Moscow Institute of Physics and Technology%
 \AND
 \Name{Alexander Gasnikov} \Email{gasnikov@yandex.ru}\\
 \addr Moscow Institute of Physics and Technology, Institute for Information Transmission Problems, National Research University Higher School of Economics
 \AND
 \Name{Alexander Tyurin} \Email{alexandertiurin@gmail.com}\\
 \addr National Research University Higher School of Economics
 \AND
 \Name{Dmitry Pasechnyuk} \Email{pasechnyuk2004@gmail.com}\\
 \addr 239-th school of St. Petersburg%
 \AND
 \Name{Artem Agafonov} \Email{agafonov.ad@phystech.edu}\\
 \addr Moscow Institute of Physics and Technology
 \AND
 \Name{Pavel Dvurechensky} \Email{pavel.dvurechensky@gmail.com}\\
 \addr Weierstrass Institute for Applied Analysis and Stochastics 
 \AND
 \Name{Darina Dvinskikh} \Email{darina.dvinskikh@wias-berlin.de}\\
 \addr Weierstrass Institute for Applied Analysis and Stochastics 
 \AND
 \gav{
 \Name{Alexey Kroshnin} \Email{gorianzwei@gmail.com}\\
 \addr Institute for Information Transmission Problems, National Research University Higher School of Economics}
 \AND
 \Name{Victorya Piskunova} \Email{piskunova\_viktoria@mail.ru}\\
 \addr V. Vernadsky Crimean Federal University%
}

\date{September 2, 2018}

\begin{document}

\maketitle

\begin{abstract}%
\pd{In this paper we propose a general algorithmic framework for first-order methods in optimization in a broad sense, including minimization problems, saddle-point problems and variational inequalities. This framework allows to obtain many known methods as a special case, the list including accelerated gradient method, composite optimization methods, level-set methods, proximal methods. The idea of the framework is based on constructing an inexact model of the main problem component, i.e. objective function in optimization or operator in variational inequalities. Besides reproducing known results, our framework allows to construct new methods, which we illustrate by constructing a universal method for variational inequalities with composite structure. This method works for smooth and non-smooth problems with optimal complexity without a priori knowledge of the problem smoothness. We also generalize our framework for strongly convex objectives and strongly monotone variational inequalities.}
\end{abstract}



\section{Introduction}
\label{S:Introduction}
We consider convex optimization problem
\begin{equation}
\label{Problem}
f(x)\to\min_{x\in Q}.
\end{equation}
It's well known (see \cite{devolder2014first, dvurechensky2017universal}) that if for all $x,y \in Q$
$$
	f(y) + \la \nabla_{\delta} f(y), x-y \ra  - \delta \le f(x)  \le f(y) + \la \nabla_{\delta} f(y), x-y \ra + \frac{L}{2}\|x-y\|^2_2 + \delta,
$$
then assuming that for proper $\alpha$ we can solve with $\tilde{\delta}$ `precision' of auxiliary problems at each iteration
$$\alpha\la \nabla_{\delta} f(y), x-y \ra  + \frac{1}{2}\|x-y\|^2_2 \to \min_{x \in Q},$$
one can prove that Gradient Method (GM) and Fast Gradient Method (FGM) converge as follows
\begin{equation}
\label{estimate}
f(x_N) - f(x_*) = O\left(\frac{LR^2}{N^p} + N^{1-p}\tilde{\delta} + N^{p-1}\delta\right),
\end{equation}
where $p=1$ for GM and $p=2$ for FGM, $x_*$~-- is a solution of \eqref{Problem}.

The \textbf{first goal}\footnote{\gav{This goal has already been realized in our previous works \cite{gasnikov2017universal,tyurin2017fast}. Here we formulate these results for completeness.}} of this paper is to show that if instead of 
function (model) $\la \nabla_{\delta} f(y), x-y \ra$ \dd{linear in $x$ } we take arbitrary 
function $\psi_{\delta}(x, y)$ (with $\psi_{\delta}(x, x) = 0$) \dd{convex in $x$} such that for arbitrary $x,y \in Q$  
$$
	f(y) + \psi_{\delta}(x, y)  - \delta \le f(x)  \le f(y) + \psi_{\delta}(x, y)+ \frac{L}{2}\|x-y\|^2_2 + \delta,
$$
\dd{then} assuming that for proper $\alpha$ we can solve with $\tilde{\delta}$ `precision' of  auxiliary problems at each iteration
$$\alpha\psi_{\delta}(x, y)  + \frac{1}{2}\|x-y\|^2_2 \to \min_{x \in Q},$$
one can prove that corresponding `model' versions of Gradient Method (GM) and Fast Gradient Method (FGM) converge with the same rates \eqref{estimate}. 
\dd{It should be noted}, that not 
\dd{every} variant of fast gradient method is well suited for such a `model's generalization'. It is significant that proper variant of FGM is based on accelerated mirror descent type of 
\dd{the} method \dd{by} \cite{tseng2008accelerated,lan2012optimal,dvurechensky2017adaptive}
\dd{which solves only one auxiliary problem of mirror descent type (not dual averaging) at each iteration.}

In particular\dd{,} \dd{as simple corollaries} these results allow to obtain 
the standard facts about the 
\dd{convergence rates}
of composite (accelerated) gradient methods \dd{presented in} \cite{beck2009fast,nesterov2013gradient} for $f(x): = g(x) + h(x)$, $\psi_{\delta}(x,y) = \langle\nabla g(y), x - y \rangle + h(x) - h(y)$
 and level (accelerated) gradient methods \dd{from} \cite{nemirovskii1985optimal, lan2015bundle} for $f(x) := g(g_1(x), \dots, g_m(x))$, $\psi_{\delta}(x,y) = g(g_1(y) + \langle\nabla g_1(y), x - y \rangle, \dots, g_m(y)+\langle\nabla g_m(y), x - y \rangle) - f(y)$.
 
 The \textbf{second goal}\footnote{\gav{The idea was proposed in \cite{gasnikov2017universal}. Here we realized this idea more generally.}} is to generalize the results mentioned above 
 \dd{to} the non-\dd{E}uclidian prox set-up. Moreover, for GM we 
 combine our model conception with the conception of relative smoothness from \cite{bauschke2016descent,lu2018relatively}. As a byproduct we 
 \dd{reproduce}  a proximal gradient method in non-\dd{E}uclidian set-up \cite{chen1993convergence} (choosing $\psi_{\delta}(x,y) =  f(x) - f(y)$). We demonstrate the 
 \dd{value of reproduced} method by applying it 
 to Wasserstein distance calculation problem \dd{with KL-prox set-up}  \cite{dvurechensky2018computational,xie2018fast,stonyakin2019gradient}.
 
 The \textbf{third goal} is to 
 \dd{supplement} the set of examples of inexact gradient oracle from \cite{devolder2014first}. In particular, we consider the following set up\footnote{\gav{This example was taken from \cite{gasnikov2015universal}.}} $f(x) := \min_{y \in Q}F(y,x)$ 
 \dd{(changing $\max$ to $\min$ in \cite{devolder2014first})}. As a byproduct of Moreau envelope smoothing example from \cite{devolder2014first} we  
 \dd{reproduce} Catalyst approach by \cite{lin2015universal}. 
 
 The \textbf{fourth goal}\footnote{\gav{We try to implement this goal based on the works \cite{dvurechensky2017adaptive,dvurechensky2018generalized,gasnikov2017universal,stonyakin2019some}.}} is to generalize the model set-up  with relative smoothness  
 \dd{to} a vector field and monotone variational inequalities (VI). We propose a proper model generalization of optimal method for VI: Mirror Prox from \cite{nemirovski2004prox}. As a byproduct this generalization allows to 
 partially \dd{reproduce} the results from \cite{chambolle2011first-order}.
 
 The \textbf{fifth goal} is to propose universal variants (see \cite{nesterov2015universal}) of the methods described above. 
 \dd{To the best of our knowledge there is no (optimal) universal method for VI even without model generality \gav{in English}.\footnote{\gav{Universal method for VI was firstly proposed Russian's book \cite{gasnikov2017universal}. In this book one can also find announcement of possibility of model generalization. In preprint (on Russian) \cite{stonyakin2019some} one can find universal model generalization.}}} 
	
The \textbf{sixth goal} is to generalize the results mentioned above for strongly convex problems and strongly monotone VI. Note, that for accelerated methods (FGM) we may use 
\dd{the} standard restart scheme, see,
 e.g. \cite{dvurechensky2017adaptive} but for non\dd{-}accelerated \dd{methods} (GM) there exists a possibility to eliminate restarts. Moreover, there exist different possibilities to determine the model conception in strongly convex case,\dd{ which we compare in this paper:} \textbf{i)} \dd{strongly convex objective $f$}; \textbf{ii)} \dd{function $\psi_{\delta}(y, x)$  strongly convex  in $y$}; \textbf{iii)} like in \cite{devolder2013firstCORE}.
	

\dd{Although the unified structure of first-order methods is not new, see, e.g. \cite{nemirovsky1983problem,mairal2013optimization,ochs2017non}, our approach generalizes only linear part of objective function approximation, that allows to combine more facts together and keeps prospects for further generalizations.}
In particular,
\dd{our proposed} model conception and corresponding GM and FGM can be considered from a primal-dual point of view as in \cite{nesterov2009primal-dual,nemirovski2010accuracy} and block-coordinate generality as in \cite{dvurechensky2017randomized}.

\section{Inexact Model for Minimization}

\subsection{Definitions and Examples}
\label{S:Prel}

We start with the general notation. Let $E$ be a finite-dimensional real vector space and $E^*$ be its dual. We denote the value of a linear function $g \in E^*$ at $x\in E$ by $\la g, x \ra$. Let $\|\cdot\|$ be some norm on $E$, $\|\cdot\|_{*}$ be its dual, defined by $\|g\|_{*} = \max\limits_{x} \big\{ \la g, x \ra, \| x \| \leq 1 \big\}$. We use $\nabla f(x)$ to denote any subgradient of a function $f$ at a point $x \in {\rm dom} f$.

Consider convex optimization problem \eqref{Problem}.

\begin{definition}	
\label{model}
	Suppose \ag{that for a given point $y \in Q$ and for all $x \in Q$ } the inequality
	\begin{equation}
	\label{model_def}
	0 \le f(x) - (f_{\delta}(y) + \psi_{\delta}(x, y)) \le LV[y](x) + \delta
	\end{equation}
	holds for some $\psi_{\delta}(x, y)$,
	$f_{\delta}(y) \in [f(y) - \delta; f(y)]$, $L$, $\delta > 0$ and $V[y](x) = d(x) - d(y) - \la \nabla d(y), x - y \ra$, where $d(x)$ is convex function on $Q$ . Let $\psi_{\delta}(x, y)$ be convex in $x \in Q$ and satisfy $\psi_{\delta}(x, x) = 0$ for all $x \in Q$. Then we say that \ag{$\psi_{\delta}(x, y)$ is  ($\delta$, L)-model of the function $f$ at a given point $y$ with respect to (w.r.t.) $V[y](x)$. }
\end{definition}

\begin{remark}
\label{Bregman}
Function $V[y](x)$, defined above as $V[y](x) = d(x) - d(y) - \la \nabla d(y), x - y \ra$
is often called Bregman divergence \cite{ben-tal2015lectures}. But typically it should be added the (1-SC) assumption in  definition: 
$d(x)$ is $1$-strongly convex on $Q$ w.r.t. $\|\cdot\|$-norm. Note that in  Definition~\ref{model} we do not need  such assumption. 
But sometimes we also use the  definition of $V[y](x)$ in the description of algorithms below and corresponding theorems of convergences rates separately. If additionally the condition (1-SC) is required we write it explicitly, see, e.g. Section~\ref{fastGradMethod}.
\end{remark}

\begin{remark}
\label{NormModel}
\label{remark_3}
\dd{We  change  `w.r.t  $V[y](x)$' to `w.r.t. $\|\cdot\|$-norm' in Definition~\ref{model} if we use $\frac{1}{2}\|x-y\|^2$ instead of $V[y](x)$.}
Typically, the (1-SC) condition (see Remark~\ref{Bregman}) on $V[y](x)$ in description of algorithms and theorem statements required below if one deal with the model w.r.t. $\|\cdot\|$-norm.
\end{remark}





\begin{remark}
Note that model definition from \dd{Remark} \ref{NormModel} is close to the definition from \cite{devolder2014first}: function $f$ has $(\delta, L)$-oracle at a given point $y$ if there exists a pair $(f_{\delta}(y), \nabla f_{\delta}(y))$ such that for all $x \in Q$: 
$0 \leq f(x) - f_{\delta}(y) - \langle\nabla f_{\delta}(y), x - y\rangle \leq \frac{L}{2}\norm{x - y}^2 + \delta$.
\end{remark} 

Now we consider some examples in which the concept of $(\delta, L)$-model of \pd{objective} function is useful. \pd{Let us} start with some standard examples.

\begin{example}\label{LG} {\bf Convex optimization problem with Lipschitz continuous gradient, \cite{nesterov2004introduction}}

If convex function $f$ has Lipschitz continuous gradient: 
\begin{equation}
\norm{\nabla f(x) - \nabla f(y)}_* \leq L\norm{x - y},\,\,\,\forall x,y \in Q.
\end{equation}
then
\begin{gather}
0 \leq f(x) - f(y) - \langle\nabla f(y), x - y \rangle \leq \frac{L}{2}\norm{x - y}^2 \,\,\, \forall x,y \in Q.
\end{gather}

In this case 
\begin{equation} 
\psi_{\delta}(x, y):= \langle\nabla f(y), x - y \rangle \,\,\, \forall x,y \in Q
\end{equation}
is $\left(0,L \right)$-model of $f$ with $f_{\delta}(y) = f(y)$ at a given point $y$ w.r.t. $\|\cdot\|$-norm.
\end{example}

\begin{example} {\bf Composite optimization, \cite{beck2009fast, nesterov2013gradient}}

Let us consider composite convex optimization problem:
\begin{align*}
f(x) := g(x) + h(x) \rightarrow \min_{x \in Q},
\end{align*}
where $g$ is a smooth convex function and the gradient of $g$ is Lipschitz continuous with parameter $L$. Function $h$ is a simple convex function. One can show
\begin{gather*}
0 \leq f(x) - f(y) - \langle\nabla g(y), x - y \rangle - h(x) + h(y) \leq \frac{L}{2}\norm{x - y}^2  ,\,\,\, \forall x,y \in Q.
\end{gather*}
Therefore
$$\psi_{\delta}(x,y) = \langle\nabla g(y), x - y \rangle + h(x) - h(y),$$
is $\left(0,L \right)$-model of $f$ with $f_{\delta}(y) = f(y)$ at a given point $y$ w.r.t. $\|\cdot\|$-norm.
\end{example}

\begin{example}{\bf Superposition of functions, \cite{nemirovskii1985optimal}}

Let us consider the following optimization problem \cite{lan2015bundle}:
\begin{align}
f(x) := g(g_1(x), \dots, g_m(x)) \rightarrow \min_{x \in Q}
\end{align}
where each function $g_k(x)$ is a smooth convex function with $L_k$-Lipschitz gradient w.r.t. $\|\cdot\|$-norm for all $k$. Function $g(x)$ is a $M$-Lipschitz convex function w.r.t 1-norm, non-decreasing \dd{in}
each of its arguments. From these assumptions we have (\cite{boyd2004convex, lan2015bundle}) that function $f(x)$ is also convex function and the following inequality holds (see \cite{lan2015bundle}):
\begin{gather*}
0 \leq f(x) - g(g_1(y) + \langle\nabla g_1(y), x - y \rangle, \dots, g_m(y)+\langle\nabla g_m(y), x - y \rangle) \leq\\\leq M\frac{\sum_{i=1}^{m}L_i}{2}\norm{x - y}^2  \,\,\,\, \forall x,y \in Q.
\end{gather*}
Also
\begin{gather*}
0 \leq f(x) - f(y) - g(g_1(y) + \langle\nabla g_1(y), x - y \rangle, \dots, g_m(y)+\langle\nabla g_m(y), x - y \rangle) + f(y) \leq\\\leq M\frac{\sum_{i=1}^{m}L_i}{2}\norm{x - y}^2  \,\,\,\, \forall x,y \in Q.
\end{gather*}
Therefore
$$\psi_{\delta}(x,y) = g(g_1(y) + \langle\nabla g_1(y), x - y \rangle, \dots, g_m(y)+\langle\nabla g_m(y), x - y \rangle) - f(y),$$
is $\left(0,M\cdot\left(\sum_{i=1}^{m}L_i\right) \right)$-model of $f$ with $f_{\delta}(y) = f(y)$ at a given point $y$ w.r.t. $\|\cdot\|$-norm. It should be note that problems \eqref{equmir2DL_G} and \eqref{equmir2DL} can be more complicated compared to traditional case when we solve smooth convex optimization problem with Lipschitz gradient. 
\end{example}

\begin{example}\label{E:Proximal}{\bf Proximal method, \cite{chen1993convergence}}

\label{prox_ex}
Let us consider optimization problem \eqref{Problem}, where $f$ is an arbitrary convex function (not necessarily smooth).
Then for arbitrary $L \ge 0$
\begin{gather*}
\psi_{\delta}(x,y) = f(x) - f(y)
\end{gather*}
is  $(0, L)$-model of $f$ with $f_{\delta}(y) = f(y)$ at a given point $y$ w.r.t $V[y](x)$, see Definition~\ref{Problem} and Remark~\ref{Bregman}. Gradient method (see\footnote{\gav{To say more precisely if we deal with proximal model (see also Remark~\ref{RemarCatalyst} and Examples~\ref{catalyst_ex},~\ref{Sinkhorn}) it is worth to use non adaptive algorithm, with fixed $L$.}} Algorithm~\ref{Alg1}) with the proposed model is equivalent to the proximal method with general Bregman divergence instead of \dd{E}uclidean one \cite{parikh2014prox}. We discus this model in more details in Appendix~\ref{ModelExampes}. In particular, based on this model (with Bregman divergence to be Kullback--Leibler divergence) and Algorithm~\ref{Alg1} we propose proximal Sinkhorn's algorithm for Wasserstein distance calculation problem (see \cite{stonyakin2019gradient}). Also we explain, what \pd{difficulties} arise \pd{in an attempt to propose} \ag{accelerated method deal with this model. The problem is that the complexity of auxiliary problems growth with the iteration number. So we introduce another model and, based on this model, we construct accelerated proximal method and show that the} \pd{Catalyst approach \cite{lin2015universal} for generic acceleration can be derived using this model.}

\end{example}

\begin{example} {\bf Min-min problem}

Consider optimization problem:
\begin{align*}
f(x) := \min_{z \in Q}F(z,x) \rightarrow \min_{x \in \R^n}.
\end{align*}
Set Q is convex and bounded.
Function F is smooth and convex w.r.t. all variables. Moreover,
$$\norm{\nabla F(z',x') -\nabla F(z,x)}_2 \leq L \norm{(z',x') -(z,x)}_2,\,\,\forall z,z'\in Q,\,x,x'\in \R^n.$$
If we can find a point $\widetilde{z}_\delta(x) \in Q$ such that
\begin{gather*}
\langle\nabla_z F(\widetilde{z}_\delta(x), x), z - \widetilde{z}_\delta(x)\rangle \geq -\delta, \,\,\,\forall z \in Q,
\end{gather*}
then
$F(\widetilde{z}_\delta(x), x) - f(x) \leq \delta$, $\norm{\nabla f(x') -\nabla f(x)}_2 \leq L \norm{x'-x}_2$
and
\begin{gather*}
\psi_{\delta}(x,y) = \langle\nabla_z F(\widetilde{z}_\delta(y), y), x - y\rangle
\end{gather*}
is  $(6\delta, 2L)$-model of $f$ with $f_\delta(y) = F(\widetilde{z}_\delta(y), y) - 2\delta$ at a given point $y$ w.r.t 2-norm.
\end{example}





\subsection{Gradient Method with Inexact Model}\label{GM}

In this section we consider a simple non-accelerated method for optimization problems with $(\delta, L)$-model. 
This method is a variant of the \ag{standard} gradient method \cite{polyak1987introduction} with adaptive tuning to the Lipschitz constant of the gradient of the objective function \cite{nesterov2013gradient}. 

We assume that on each iteration $k$, the method has access to $(\delta_k, L)$-model of $f$ \ag{w.r.t. $V[y](x)$} 
(see Definition~\ref{model}).  
Depending on the problem, $\delta_k$ can  be equal to zero,  constant value or change from iteration to iteration. 

\begin{algorithm}
\caption{Gradient method with an oracle using the $(\delta, L)$-model}
\label{Alg1}
\begin{algorithmic}[1]
\STATE \textbf{Input:} $x_0$ is the starting point, 
$\{\delta_k\}_{k\geq 0}$ and
$ L_{0} > 0$.
\STATE Set
 $\alpha_0 := 0$, $A_0 := \alpha_0$
\FOR{$k \geq 0$}
\STATE Find the smallest $i_k\geq 0$ such that
\begin{equation}\label{exitLDL_G}
f_{\delta_k}(x_{k+1}) \leq f_{\delta_k}(x_{k}) + \psi_{\delta_k}(x_{k+1}, x_{k}) +L_{k+1}V[x_{k}](x_{k+1}) + \delta_k,
\end{equation}
where $L_{k+1} = 2^{i_k-1}L_k$, $\alpha_{k+1} := \frac{1}{L_{k+1}}$, $A_{k+1} := A_k + \alpha_{k+1}$.
\begin{equation}\label{equmir2DL_G}
\phi_{k+1}(x)= 
\gav{L_{k+1}\cdot}\left(V[x_k](x) + \gav{\alpha_{k+1}}\psi_{\delta_k}(x, x_{k})\right), \quad
x_{k+1} := \arg\min_{x \in Q} \phi_{k+1}(x).
\end{equation}

\ENDFOR
\ENSURE 
\dd{$\bar{x}_N= \frac{1}{A_N}\sum_{k=0}^{N-1}\alpha_{k+1} x_{k+1}$}
\end{algorithmic}
\end{algorithm}





\begin{theorem}
	\label{mainTheoremDL_G}
		Let $V[x_0](x_*) \leq R^2$, where $x_0$~ is the starting point, and $x_*$~ is the nearest minimum point to the point $x_0$ in the sense of Bregman divergence (see Remark~\ref{Bregman}). \pd{Then, for the sequence, generated by  Algorithm~\ref{Alg1} } 
		\dd{ the following holds}
	\begin{equation}\label{grad_estim_thm}
	f(\bar{x}_N) - f(x_*) \leq \frac{R^2}{A_N}  + \frac{2}{A_N}\sum_{k=0}^{N-1}\alpha_{k+1}\delta_k \leq \frac{2LR^2}{N} + \frac{2}{A_N}\sum_{k=0}^{N-1}\alpha_{k+1}\delta_k,
	\end{equation}
	where $A_N 
\geq \frac{N}{2L}$.
 Moreover, the total number of attempts to solve \eqref{equmir2DL_G} is bounded by $2N + \log_2\frac{L}{L_0}$. 
	\end{theorem}
	
\begin{remark}
\label{wrtN}
\ag{If the $(\delta_k, L)$-model of $f$ is given w.r.t. $\|\cdot\|$-norm (see Remark~\ref{NormModel}), then the chosen  $V[y](x)$ in Algorithm~\ref{Alg1} and Theorem~\ref{mainTheoremDL_G} has to satisfy (1-SC) condition w.r.t. this norm (see Remarks~\ref{Bregman},~\ref{NormModel}).}
\end{remark}
	



\subsection{
Fast Gradient Method with Inexact Model
} \label{fastGradMethod}


In this section we consider accelerated method for problems with $(\delta, L)$-model. The method is close to accelerated mirror-descent type of methods by \cite{tseng2008accelerated,lan2012optimal,dvurechensky2018computational}. On each iteration, the inexact model is used to make a mirror-descent-type of step.
\ag{In this section, we assume that the $(\delta_k, L)$-model of $f$ is given w.r.t. $\|\cdot\|$-norm and  $V[u](x)$ satisfies (1-SC) condition w.r.t. this norm (see Remarks~\ref{Bregman},~\ref{NormModel},~\ref{wrtN})}.


\begin{algorithm}
\caption{\bf{Fast gradient method with oracle using $(\delta, L)$-model}}
\label{Alg2}
\begin{algorithmic}[1]
\STATE \textbf{Input:} $x_0$~ is the starting point, $\{\delta_k\}_{k\ge0}$ and $L_0 > 0$. 
\STATE Set
$y_0 := x_0$, $u_0 := x_0$, $\alpha_0 := 0$, $A_0 := \alpha_0$
\FOR{$k \geq 0$}
\STATE Find the smallest $i_k \geq 0$ such that
\begin{equation}
\begin{gathered}
f_{\delta_k}(x_{k+1}) \leq f_{\delta_k}(y_{k+1}) + \psi_{\delta_k}(x_{k+1}, y_{k+1}) +\frac{L_{k+1}}{2}\norm{x_{k+1} - y_{k+1}}^2 + \delta_k,
\label{exitLDL}
\end{gathered}
\end{equation}
where $L_{k+1}=2^{i_k-1}L_k$, $\alpha_{k+1}$ is the largest root  of  $A_{k+1}=L_{k+1}\alpha^2_{k+1}$, $A_{k+1} := A_k + \alpha_{k+1}$.
\begin{gather}
y_{k+1} := \frac{\alpha_{k+1}u_k + A_k x_k}{A_{k+1}} \label{eqymir2DL}
\end{gather}
\begin{equation}\label{equmir2DL}
\phi_{k+1}(x)=\at2{L_{k+1}}\gav{\cdot}(V[u_k](x) + \alpha_{k+1}\psi_{\delta_k}(x, y_{k+1})),\quad
u_{k+1} := \argmin_{x \in Q}\phi_{k+1}(x)
\end{equation}
\begin{gather}
x_{k+1} := \frac{\alpha_{k+1}u_{k+1} + A_k x_k}{A_{k+1}} \label{eqxmir2DL}
\end{gather}

\ENDFOR
\ENSURE $x_N$
\end{algorithmic}
\end{algorithm}





\begin{theorem}
	\label{mainTheoremDL}
	Let $V[x_0](x_*) \leq R^2$, where $x_0$~ is the starting point and $x_*$~ is the nearest minimum point to $x_0$ in the sense of Bregman divergence. 
	\pd{Then, for the sequence, generated by  Algorithm~\ref{Alg2}, }
	\begin{equation*}
	f(x_N) - f(x_*) \leq \frac{R^2}{A_N} + \frac{2\sum_{k=0}^{N-1}\delta_kA_{k+1}}{A_{N}}  \leq \frac{8LR^2}{(N+1)^2} + \frac{2\sum_{k=0}^{N-1}\delta_kA_{k+1}}{A_{N}}, 
	\end{equation*}
	where $A_k \geq \frac{(k+1)^2}{8L}$. Moreover, the total number of attempts to solve \eqref{equmir2DL} is bounded by $4N + \log_2\frac{L}{L_0}$.  
\end{theorem}

\section{Inexact Model for Variational Inequalities}
\label{VI}

In this section, we go beyond minimization problems and propose an abstract inexact model counterpart for variational inequalities. Using this model we propose a new universal method for variational inequalities with complexity $
O\left(\inf_{\nu \in [0,1]}\left(\frac{1}{\varepsilon}\right)^{\frac{2}{1+\nu}}\right)$, where $\varepsilon$ is the desired accuracy of the solution. 
According to the lower bounds in \cite{Optimal}, this algorithm is optimal for $\nu = 0$ and $\nu = 1$.
Based on the model for VI and functions, we extend $(\delta, L)$-model for saddle-point problems (see Appendix~\ref{SP}).
We are also motivated by  mixed variational inequalities \cite{Konnov_2017, Bao_Khanh} and composite saddle-point problems \cite{chambolle2011first-order}. 

We consider the problem of finding the solution $x_*\in Q$ for VI in the following abstract form
\begin{equation}\label{eq13}
\psi(x,x_*)\geqslant 0 \quad \forall x \in Q
\end{equation}
for some convex compact set $Q\subset\mathbb{R}^n$ and some function $\psi:Q\times Q\rightarrow\mathbb{R}$. Assuming the abstract monotony of the function $\psi$ 
\begin{equation}\label{eq:abstr_monot}
\psi(x,y)+\psi(y,x)\leqslant0\;\;\;\forall x,y\in Q,
\end{equation}
any solution \eqref{eq13} will is a solution of the following inequality 
\begin{equation}\label{eq115}
\max_{x\in Q}\psi(x_*,x)\leqslant 0  \quad \forall x \in Q.
\end{equation}

In the general case, we make an assumption about the existence of a solution $x_*$ of the problem \eqref{eq13}. As a particular case, if for some operator $g: Q \rightarrow\mathbb{R}^n$ we set $\psi(x,y)=\langle g(y),x-y\rangle\;\;\forall x,y\in Q$,
then \eqref{eq13} and \eqref{eq115} are equivalent, respectively, to a standard strong and weak variational inequality with the operator $g$.

\begin{example}
For some operator $g:Q\rightarrow\mathbb{R}^n$ and a convex functional $h:Q\rightarrow\mathbb{R}^n$ choice
\begin{equation}\label{eq17}
\psi(x,y)=\langle g(y),x-y\rangle+h(y)-h(x)
\end{equation}
leads to a {\it mixed variational inequality from} \cite{Konnov_2017, Bao_Khanh}
\begin{equation}\label{eq18}
\langle g(y),y-x\rangle+h(x)-h(y)\leqslant 0,
\end{equation}
which in the case of the monotonicity of the operator $g$ implies
\begin{equation}\label{eq19}
\langle g(x),y-x\rangle+h(x)-h(y)\leqslant 0.
\end{equation}
\end{example}

We propose an adaptive proximal method for the problems \eqref{eq13} and \eqref{eq115}. We start with a concept of $(\delta, L)$-model for such problems.

\begin{definition}\label{Def_Model_VI}
We say that functional $\psi$ has  $(\delta, L)$-model $\psi_{\delta} (x, y)$ at a given point $y$ w.r.t. $V[y](x)$ if the following properties hold for each $x, y, z \in Q$:
\begin{enumerate}
\item[(i)] $\psi_{\delta} (x, y)$ convex in the first variable; 
\item[(ii)] $\psi_{\delta}(x,x)=0$;
\item[(iii)] ({\it abstract $\delta$-monotonicity}) 
\begin{equation}\label{eq:abstr_delta_monot}
\psi_{\delta}(x,y)+\psi_{\delta}(y,x)\leq \delta;
\end{equation}
\item[(iv)] ({\it generalized relative smoothness}) 
\begin{equation}\label{eq20}
\psi_{\delta}(x,y)\leqslant\psi_{\delta}(x,z)+\psi_{\delta}(z,y)+ LV[z](x)+ LV[y](z)+\delta
\end{equation}
for some fixed values $L>0$, $\delta>0$.
\end{enumerate}
\end{definition}

\begin{remark}
\label{BregmanVI}
Similarly to Definition~\ref{model} above, in general case, we do not need the (1-SC) assumption in Definition \ref{Def_Model_VI} for $V[y](x)$.
In some situations we assume that (1-SC) assumption holds (see Examples \ref{Examp_Smooth_VI}, \ref{UMP_Example} and Appendix \ref{UMPStrongApp}). 
\end{remark}

\begin{remark}
\label{NormModelVI}
In Definition \ref{Def_Model_VI} we change `w.r.t  $V[y](x)$' to `w.r.t. $\|\cdot\|$-norm if we use $\frac{1}{2}\|x-y\|^2$ instead of $V[y](x)$.
\end{remark}

Note that for $\delta=0$ the following analogue of \eqref{eq20} for some fixed $a, b > 0$
\begin{equation}\label{eq200}
\psi(x,y)\leqslant\psi(x,z)+\psi(z,y)+ a\|z - y\|^2 + b\|x - z\|^2 \quad  \forall x, y, z \in Q
\end{equation}
was introduced in \cite{Mastroeni}. Condition \eqref{eq200} is used in many works on equilibrium programming. Our approach allows us to work with non-Euclidean set-up without (1-SC) assumption and inexactness $\delta$, that is important for the ideology of universal methods \cite{nesterov2015universal} (see Example \ref{UMP_Example} below).

One can directly verify that if $\psi_{\delta}(x,y)$ is  $(\delta/5, L)$-model of the function $f$ at a given point $y$ w.r.t. $V[y](x)$ then $\psi_{\delta}(x,y)$ is  $(\delta, L)$-model in the sense of Definition \ref{Def_Model_VI}  w.r.t. $V[y](x)$.\\

Let us consider some examples.

\begin{example}\label{Examp_Smooth_VI} {\bf Variational Inequalities with monotone Lipshitz continuous operator.}
Consider variational inequality of finding $x \in Q$ such that $\la g(y),x-y\ra \leq 0$, $\forall y \in Q$, the operator $g: Q \rightarrow R^{n}$ is monotone and Lipschitz continuous, i.e.
$ \norm{g(x) - g(y)}_* \leq L\norm{x - y},\,\,\,\forall x,y \in Q.$
In this case $
\psi_{\delta}(x, y):= \langle g(y), x - y \rangle $ is a ($\delta$, L)-model in a sense of Definition \ref{Def_Model_VI} w.r.t. $\|\cdot\|$-norm ($\forall x,y \in Q$). 
\end{example}

\begin{example}\label{UMP_Example} {\bf Variational Inequalities with monotone Holder continuous operator.} 
\label{example_universal_g}
Assume that for monotone operator $g$ there exists $\nu\in[0,1]$ such that
\begin{equation}\label{Hold_cont_g}
\norm{g(x) - g(y)}_* \leq L_{\nu}\norm{x -  y}^\nu,\,\,\,\forall x,y \in Q.
\end{equation}

Then we have: $
\langle g(z)-g(y), z-x\rangle\leq 
\|g(z) - g(y) \|_* \|z-x\| \leq
$
\begin{equation}\label{Hold_interpol}
    \leq  L_{\nu}\|z-y\|^{\nu}  \|z-x\| \leq \frac{L(\delta)}{2}||z-x||^2+\frac{L(\delta)}{2}||z-y||^2+\delta 
\end{equation}
for
\begin{equation}\label{UMP_constant} L(\delta)  = \left(\frac{1}{2\delta}\right)^\frac{1-\nu}{1+\nu} L_{\nu}^{\frac{2}{1+\nu}}
\end{equation} 
and uncontrolled parameter $\delta > 0$.  In this case the following function  
\begin{equation}\label{UMP_model}
\psi_{\delta}(x, y):= \langle g(y), x - y \rangle \,\,\, \forall x,y \in Q.
\end{equation}
is ($\delta$, L)-model w.r.t. $\|\cdot\|$-norm. 
\end{example}

Note that for the previous two examples in Algorithm \ref{Alg:UMPModel} and Theorem~\ref{thmm1} we need $V[z](x)$ to satisfy (1-SC) condition.

Next, we introduce our novel adaptive method (Algorithm \ref{Alg:UMPModel}
) for abstract variational inequalities with inexact $(\delta, L)$-model\footnote{Here we assume that $\delta$ doesn't change on iterations. We allow $\delta$ to change before (e.g. in Section~\ref{fastGradMethod}) for possibility to build universal fast gradient method, see Example~\ref{universal}. But for non accelerated methods it is not necessary. In Section~\ref{GM} we, actually, change $\delta$ on iteration for the convenience of comparison the results of Sections~\ref{GM} and~\ref{fastGradMethod}.} w.r.t. $V[y](x)$. If $V[y](x)$  satisfies (1-SC) condition then we can consider inexact $(\delta, L)$-model w.r.t. $\|\cdot\|$-norm. This method adapts to the local values of $L$ and similarly to \cite{nesterov2015universal} allows us to construct universal method for variational inequalities. Applying the following adaptive Algorithm \ref{Alg:UMPModel} to VI with Holder interpolation \eqref{Hold_interpol} 
for $\delta = \frac{\varepsilon}{2}$ and $L = L\left(\frac{\varepsilon}{2}\right)$ leads us to universal method for VI.

\begin{algorithm}[ht]
\caption{Generalized Mirror Prox for VI}
\label{Alg:UMPModel}
\begin{algorithmic}[1]
   \REQUIRE accuracy $\e > 0$, oracle error $\delta >0$,
   initial guess $L_{0} >0$, 
   prox set-up: $d(x)$, $V[z] (x)$.
   \STATE Set $k=0$, $z_0 = \arg \min_{u \in Q} d(u)$.
   \FOR{$k=0,1,...$}
				\STATE Find the smallest $i_k \geq 0$ such that 
				\begin{equation}\label{eqUMP23}
                \begin{split}
                \hspace{-2em}\psi_{\delta}(z_{k+1}, z_{k})\leq \psi_{\delta}(z_{k+1}, w_{k})+\psi_{\delta}(w_k,z_k)+ L_{k+1}(V[z_k](w_k) + V[w_k](z_{k+1})) + \delta,
                \end{split}
                \end{equation}
			where 	$L_{k+1}=2^{i_k-1}L_{k}$ and 
			\begin{align}
			    w_k&={\mathop {\arg \min }\limits_{x\in Q}} \left\{\psi_{\delta}(x, z_k)+ L_{k+1}V[z_k](x) \right\}.
				\label{eq:UMPwStepMod} \\
				z_{k+1}&={\mathop {\arg \min }\limits_{x\in Q}} \left\{\psi_{\delta}(x, w_k) + L_{k+1}V[z_k](x) 		 \right\}. \label{eq:UMPzStepMod}
			\end{align}
					\ENDFOR
		\ENSURE $\widehat{w}_N = \frac{1}{\sum_{k=0}^{N-1}\frac{1}{L_{k+1}}}\sum_{k=0}^{N-1}\frac{1}{L_{k+1}}w_k$. 
	
\end{algorithmic}
\end{algorithm}

For a given accuracy $\varepsilon$ we can consider the following stopping criterion for Algorithm \ref{Alg:UMPModel}:
\begin{equation*}
S_N:= \sum_{k=0}^{N-1}\frac{1}{L_{k+1}}\geqslant \frac{V[x^0](x_*)}{\varepsilon}.
\end{equation*}

Let us formulate the following result

\begin{theorem}\label{thmm1}
For Algorithm \ref{Alg:UMPModel} the following inequalities hold 

\begin{equation*}
\max\limits_{u\in Q}\left(-\frac{1}{S_N}\sum_{k=0}^{N-1}\frac{\psi_{\delta}(u,w_{k})}{L_{k+1}} \right)\leq \frac{2L \max_{u\in Q}V[z_0](u)}{N} + \delta,
\end{equation*}
\begin{equation*}
\max\limits_{u \in Q}\psi_{\delta}(\widehat{w}_N,u)
\leq \frac{2L \max_{u\in Q}V[z_0](u)}{N} + 2\delta,
\end{equation*}
\begin{equation*}
\widehat{w}_N:=\frac{1}{S_N}\sum_{k=0}^{N-1}\frac{w_{k}}{L_{k+1}}.
\end{equation*}
\end{theorem}

\begin{proof}
After $(k+1)$-th iteration ($k=0,1,2,\ldots$) we have for each $u \in Q$:
$$\psi_{\delta}(w_k, z_k)\leqslant\psi(u, z_k) +L_{k+1}V[z_k](u)-L_{k+1}V[w_k](u)- L_{k+1}V[z_k](w_k)$$
and
$$\psi_{\delta}(z_{k+1}, w_k)\leq\psi_{\delta}(u, w_k)+L_{k+1}V[z_k](u)-L_{k+1}V[z_{k+1}](u)-L_{k+1}V[z_k](z_{k+1}).$$
Taking into account \eqref{eqUMP23}, we obtain
$$
-\psi_{\delta}(u, w_k) \leq L_{k+1}V[z_k](u)-L_{k+1}V[z_{k+1}](u) + \delta.
$$
So, the following inequality
$$-\sum_{k=0}^{N-1} \frac{\psi_{\delta}(u,w_{k})}{L_{k+1}} \leq V[z_0](u) - V[z_k](u) + S_N \delta $$
holds. By virtue of \eqref{eq20} and the choice of $L_{0}\leqslant 2L$, it is guaranteed that $$L_{k+1}\leqslant 2L\;\;\forall k=0,...,N-1.$$
and we have
\begin{equation}
\max\limits_{u \in Q} \psi_{\delta}(\widehat{w}_N ,u)\leqslant-\frac{1}{S_N}\sum_{k=0}^{N-1}\frac{\psi_{\delta}(u,w_{k})}{L_{k+1}} + \delta \leq \frac{2L \max_{u\in Q}V[z_0](u)}{N} + 2\delta.
\end{equation}
\end{proof}
\begin{remark}
To obtain precision $\varepsilon + \delta$ Algorithm \ref{Alg:UMPModel} works no more than
\begin{equation}\label{eqv_5.5}
\left\lceil\frac{2L \max_{u\in Q}V[z_0](u)}{\varepsilon}\right\rceil
\end{equation}
iterations. Note that estimate \eqref{eqv_5.5} is optimal for variational inequalities and saddle-point problems \cite{Optimal}. 

For universal method to obtain precision $\varepsilon$ we can choose $\delta = \frac{\varepsilon}{2}$ and $L = L\left(\frac{\varepsilon}{2}\right)$ according to \eqref{Hold_interpol} and \eqref{UMP_constant} and the estimate \eqref{eqv_5.5} reduces to 
\begin{equation}
\left\lceil 2 \inf_{\nu\in[0,1]}\left(\frac{2L_{\nu}}{\e} \right)^{\frac{2}{1+\nu}} \cdot \max_{u\in Q}V[z_0](u)\right\rceil.
\end{equation}

Note that similarly to Algorithms \ref{Alg1} and \ref{Alg2}, the total number of attempts to solve \eqref{eq:UMPwStepMod} and \eqref{eq:UMPzStepMod} is bounded by $4N + \log_2\frac{L}{L_0}$.
\end{remark}

Thus, the introduced concept of the function model for variational inequalities allows us to extend the previously proposed universal method for VI \dd{to} a wider class of problems, including {\it mixed variational inequalities} \cite{Konnov_2017, Bao_Khanh} and {\it composite saddle-point problems} \cite{chambolle2011first-order}. We extend $(\delta, L)$-model for saddle-point problems in Appendix~\ref{SP} further. 

\section{Concluding remarks}

\label{concluding_remarks}
Firstly, note that for all considered methods we may also take into account inexactness for auxiliary problems using the following
\begin{definition}
\label{solNemirovskiy}
For a convex optimization problem
\begin{gather}
\label{Ax_problem}
\Psi(x) \rightarrow \min_{x \in Q},
\end{gather}
we denote by $\text{Arg}\min_{x \in Q}^{\widetilde{\delta}}\Psi(x)$~a collection of $\widetilde{x}$:
\begin{gather}\label{eqv_inex_sol}
\gav{\exists} h \in \partial\Psi(\widetilde{x})\gav{:\forall x \in Q} \,\, \gav{\to}\, \langle h, x - \widetilde{x} \rangle \geq -\widetilde{\delta}. 
\end{gather}
Let \ag{us} denote by $\arg\min_{x \in Q}^{\widetilde{\delta}}\Psi(x)$  some element of $\text{Arg}\min_{x \in Q}^{\widetilde{\delta}}\Psi(x)$.
\end{definition}

\gav{Note, that if $\Psi(x)$ is $\mu$-strongly convex; has $L$-Lipschitz continuous gradient in $\|\cdot\|$ norm\footnote{To say more precisely $$L=\max_{\|h\|\le 1, x \in[\widetilde{x},x_*]}\la h,\nabla^2 \Psi(x)h \ra.$$}
and $R = \max_{x,y\in Q} \|x-y\|$, then  $\Psi(\widetilde{x}) - \Psi(x_*)\le\widetilde{\epsilon}$ entails that \cite{stonyakin2019gradient}
\begin{gather}\label{inexact}
\widetilde{\delta}\le (LR+\|\nabla\Psi(x_*)\|_*)\sqrt{2\tilde{\e}/\mu},
\end{gather}
where $x_* = \argmin_{x\in Q}\Psi(x)$.}
\gav{If one can guarantee that $\nabla\Psi(x_*) = 0$, then \eqref{inexact} can be improved $$\widetilde{\delta}\le R\sqrt{2L\widetilde{\e}}.$$ }

Clearly, for the case $\widetilde{\delta} = 0$ equation \eqref{eqv_inex_sol} means that $\widetilde{x}$ is an exact solution of \eqref{Ax_problem}. In Appendices \ref{ProofGM}, \ref{profFGM}, \ref{SC}, \ref{UMPStrongApp}  we  show that inexactness for auxiliary problems \eqref{equmir2DL_G}, \eqref{equmir2DL}, \eqref{eq:UMPwStepMod}, \eqref{eq:UMPzStepMod}
according to Definition\ref{solNemirovskiy} changes the estimates of the rate of convergence in all the methods no more than by additive term $O(\widetilde{\delta})$, e.g. see \eqref{estimate} for problem~\eqref{Problem}. 
\dd{Similarly}, in Appendix \ref{UMPInexact} and \ref{SP} for variational inequalities (VI) with monotone Lipshitz continuous operator we obtain
$$
\max\limits_{u \in Q} \langle g(u), \widehat{w}_N - u \rangle = O\left(\frac{LR^2}{N} + \widetilde{\delta} + \delta \right).
$$
and for convex-concave saddle-point problems of finding 
$\min_{u \in Q_1} \max_{v \in Q_2} f(u, v)$ we have
$$
\max_{v \in Q_2} f(\widehat{u}_N, v) - \min_{u \in Q_1} f(u, \widehat{v}_N) = O\left(\frac{LR^2}{N} +  \widetilde{\delta} + \delta \right).
$$
Secondly, note that in the case of  $\mu$-strongly convex objective (model) the estimates for the proposed minimization methods can be improved. In the same way, this also applies to the method for (VI) in the case of the strong monotonicity of the operator (model).
Details are described in appendices \ref{SC} and \ref{UMPStrongApp}. In all the cases by restart procedure  from \eqref{estimate}, one can obtain a linear rate of convergence, e.g. for problem \eqref{Problem} we get the following improved variant of \eqref{estimate} ($\Delta f = f(x^0) - f(x_*)$):
$$f(x_N) - f(x_*) = \tilde{O}\left(\Delta f \exp\left(-O(1)\left(\frac{\mu}{L}\right)^{\frac{1}{p}}N\right) + \left(\frac{L}{\mu}\right)^{\frac{1-p}{2}}\tilde{\delta} + \left(\frac{L}{\mu}\right)^{\frac{p-1}{2}}\delta  \right),$$
where $p=1$ for GM and $p=2$ for restarted FGM.

Finally, all the methods considered in this paper have universal (see \cite{nesterov2015universal}) extensions which allow to solve smooth and non-smooth problems without the prior knowledge of the smoothness level of the problem (Example~\ref{universal}).

\gav{This paper is a full English version of our results, that was written on Russian  \cite{gasnikov2017universal,tyurin2017fast}. In this paper we also add new results concerning `model' generalization of VI and generalization all the results to strongly convex case (in \cite{gasnikov2017universal} such a possibility was only announced). We also add some examples.}

\acks{The authors are grateful to Yurii Nesterov for fruitful discussions. 

The work of F.~Stonyakin on model of vector field and Universal Mirror Prox for this field was supported by Russian Science Foundation according to the research project 18-71-00048, the work of A. Gasnikov on the conception of model of function at a given point and GM with relative smoothness context was supported by RFBR 18-31-20005 mol$\_$a$\_$ved, the work of P. Dvurechensky on literature survey and general structure of the paper was supported by RFBR 18-29-03071 mk, the work of A. Tyurin in model's FGM was prepared within the framework of the HSE University Basic Research Program and funded by the Russian Academic Excellence Project '5-100', the work of D. Pasechnyk and D. Dvinskikh on proximal Sinkhorn method was fulfilled in July 2018 in Sirius (Sochi).}

\bibliography{PD_references}

\appendix

\section{Model examples} \label{ModelExampes}
In this appendix we present different examples of a $(\delta, L)$-model of objective $f$.

\begin{example}{\bf Saddle point problem, \cite{devolder2014first}}

Let us consider
\begin{align}
f(x) = \max_{z \in Q}\left[\langle x, b - Az\rangle - \phi(z)\right] \rightarrow \min_{x \in \R^n},
\end{align}
where $\phi(z)$ is a $\mu$-strong convex function w.r.t. $p$-norm ($1\leq p\leq2$). Then $f$ is a smooth convex function and the gradient of $f$ is Lipschitz continuous with parameter
\begin{gather*}
L = \frac{1}{\mu}\max_{\norm{z}_p\leq1}\norm{Az}_2^2.
\end{gather*}
If $z_\delta(y)\in Q$ is a solution of auxiliary max-problem in the following sense 
$$\max_{z \in Q}\left[\langle y, b - Az\rangle - \phi(z)\right] - \left[\langle y, b - Az_{\delta}(y)\rangle - \phi(z_{\delta}(y))\right]\le \delta,$$
then
\begin{gather*}
\psi_{\delta}(x,y) =\langle b - Az_\delta(y), x - y\rangle
\end{gather*}
is  $(\delta, 2L)$-model of $f$ with $$f_\delta(y) = \langle y, b - Az_\delta(y)\rangle - \phi(z_\delta(y))$$ at the point $y$ w.r.t 2-norm.
\end{example}

\begin{example}
{\bf Augmented Lagrangians, \cite{devolder2014first}}

Let us consider
\begin{gather*}
\phi(z) + \frac{\mu}{2}\norm{Az -b}_2^2 \rightarrow \min_{Az=b,\,z\in Q}.
\end{gather*}
and it's dual problem
\begin{align*}
f(x) = \max_{z \in Q}\underbrace{\left(\langle x, b - Az\rangle - \phi(z) - \frac{\mu}{2}\norm{Az -b}_2^2\right)}_{\Lambda(x,z)}\rightarrow \min_{x \in \R^n}.
\end{align*}
If $z_\delta(y)$ is a solution of auxiliary max-problem in the following sense
\begin{gather*}
\max_{z \in Q}\left\langle\nabla_z\Lambda(y,z_\delta(y)),z - z_\delta(y)\right\rangle \leq \delta,
\end{gather*}
then
\begin{gather*}
\psi_{\delta}(x,y) =\langle b - Az_\delta(y), x - y\rangle
\end{gather*}
is  $(\delta, \mu^{-1})$-model of $f$ with $$f_\delta(y) = \langle y, b - Az_\delta(y)\rangle - \phi(z_\delta(y))- \frac{\mu}{2}\norm{Az_\delta(y) -b}_2^2$$ at the point $y$ w.r.t 2-norm.
\end{example}

\begin{example}
{\bf Moreau envelope of target function, \cite{devolder2014first}}
\label{Moreau_ex}

Let us consider optimization problem:
\begin{align}
f_L(x) := \min_{z \in Q}\underbrace{\left\{f(z) + \frac{L}{2}\norm{z - x}^2_2\right\}}_{\Lambda(x,z)}\rightarrow \min_{x \in \R^n}.
\label{prox}
\end{align}
Assume that function $f$ is a convex function and
\begin{gather*}
\max_{z \in Q}\left\{\Lambda(y,z_L(y)) - \Lambda(y,z) + \frac{L}{2}\norm{y - z_L(y)}^2_2\right\} \leq \delta.
\end{gather*}
Then
\begin{gather*}
\psi_{\delta}(x,y) =\langle L(y - z_L(y)), x - y\rangle
\label{prox_model}
\end{gather*}
is $(\delta, L)$-model of $f$ with $$f_\delta(y) = f(z_L(y)) + \frac{L}{2}\norm{z_L(y) - y}^2_2 - \delta$$ at the point $y$ w.r.t 2-norm.
\end{example}

\begin{remark}\label{RemarCatalyst}
In paper \cite{lin2015universal} authors propose generic acceleration scheme (Catalyst) for large class of optimization problems. They replace a function from optimization problem~ \eqref{Problem} $f$ with more well-defined functions (Moreau envelop of $f$, see Example~\ref{Moreau_ex}) and apply accelerated proximal method. In our approach with $(\delta, L)$-model we can try to use proximal model from example \ref{prox_ex}. However, due to the linear growth of $\alpha_k \sim k$ in a fast gradient method our auxiliary optimization problems would be ill-conditioned. We can overcome this problem using different approach which naturally combines with $(\delta, L)$-model concept. In example \ref{catalyst_ex} we demonstrate this approach which relies heavily on example \ref{Moreau_ex}.
\end{remark}

\begin{example}
{\bf Catalyst acceleration, \cite{lin2015universal}}
\label{catalyst_ex}

Let us assume that function $f$ is  $\mu_f$-strongly convex function with $L_f$-Lipschitz gradient w.r.t 2-norm. Let us replace optimization problem~\eqref{Problem} on problem~\eqref{prox}. 
These replacement gives us the following:
\begin{enumerate}
    \item 
    There is a `closed-form' solution of the auxiliary optimization problem \eqref{equmir2DL_G} and \eqref{equmir2DL}. For instance, using $(\delta, L)$-model from \eqref{prox_model} we can show for auxiliary optimization problem from \eqref{equmir2DL} that (assume that $V[u_k](x) = \frac{1}{2}\|x-y_k\|_2^2$)
    \begin{equation}
    \phi_{k+1}(x)= \at2{L_{k+1}}\left(\frac{1}{2}\|x-u_k\|_2^2 + \at2{\alpha_{k+1}}\psi_{\delta_k}(x, y_{k+1})\right),\;
    u_{k+1} := \argmin_{x \in Q}\phi_{k+1}(x)
    \end{equation}
    is equivalent to
    \begin{align*}
        u_{k+1} :=  u_k - \alpha_{k+1}L(y_{k+1} - z_L(y_{k+1})).
    \end{align*}
    \item 
    In order to find $z_L(y_{k+1})$ we should solve `new' auxiliary optimization problem
    \begin{align*}
    f(z) + \frac{L}{2}\norm{z - y_{k+1}}^2_2 \rightarrow \min_{z \in Q},
    \end{align*}
    which is well-defined with $(\mu_f + L)$-strongly convex function and $(L_f + L)$-Lipschitz gradient w.r.t 2-norm.
\end{enumerate}
Philosophically these approach is very close to approach from \cite{lin2015universal}. 
The problem is that instead of function $f$ we minimize function $f_L$. However, we can use strong convexity of function $f$ to get around this. For simplicity, let us take $Q = \R^n$.
It can be shown \cite{polyak1987introduction} that $$f_L(x_N) - f_L(x_*) \leq f(x_N) - f(x_*),$$ where $x_*$ is an optimal solution of optimization problem \eqref{prox}. Using the fact (\cite{lemarechal1997practice}) that function $f_L$ has strong convexity parameter equal to
\begin{align}
\label{catalyst_new_mu}
\mu_L := \mu_f\frac{L_f}{\mu_f + L_f} \geq \frac{\mu_f}{2}
\end{align}
we can show that
$$\frac{\mu_L}{2}\norm{x_N - x_*}^2 \leq f_L(x_N) - f_L(x_*) \leq f(x_N) - f(x_*) \leq \frac{L_f}{2}\norm{x_N - x_*}^2.$$
Also we should note that function $f_L$ has $L$-Lipschitz gradient, we need it further. We obtain that an $\e$-solution of optimization problem \eqref{prox} is an $\e$-solution of optimization problem \eqref{Problem} with the same accuracy up to constant multiplier:
\begin{align*}
f(x_N) - f(x_*) \leq \frac{\mu_f + L_f}{\mu_f}\left(f_L(x_N) - f_L(x_*)\right).
\end{align*}

Let us assume that we solve
auxiliary optimization problem with a non-accelerated gradient method for strong convex functions (e.g. standard gradient method) with accuracy $O(\e^2)$, where $\e$ -- is desired relative accuracy by function for original problem. For external optimization method we can take FGM for smooth $\mu$-strongly convex functions with $L$-Lipschitz gradient\footnote{Restarted Algorithm~\ref{Alg2} (see Appendix~\ref{SC}) in model environment of Example~\ref{Moreau_ex}.}. We know that for this method the number of steps is equal to $O(\sqrt{L / \mu}\ln(1/\e))$ (follows from Example~\ref{Moreau_ex}). The total number of gradient calculations equals to number of steps of external optimization method multiplied by number of steps of non-accelerated gradient method. Therefore, the total number of gradient calculations equals to
$$O\left(\frac{L_f + L}{\mu_f + L} \sqrt{\frac{L}{\mu_L}}\right)$$
where constant $L$ is a free parameter. Let us take $L = L_f$. Using \eqref{catalyst_new_mu} we have that the total number of gradient calculations equals to
$$\tilde{O}\left(\frac{L_f + L_f}{\mu_f + L_f} \sqrt{\frac{L_f}{\mu_L}}\right) = \tilde{O}\left(\sqrt{\frac{L_f}{\mu_f}}\right).$$ This means that we have accelerated convergence rate for optimization problem \eqref{Problem}. In general, this approach, based on Example \ref{Moreau_ex}, allows to accelerate non-accelerated different methods.
\end{example}

\begin{example}
{\bf Proximal Sinkhorn method}
\label{Sinkhorn}

Optimal transport (OT) \cite{monge1781memoire,kantorovich1942translocation} is currently generating an increasing attraction in statistics and machine learning communities \cite{bigot2012consistent,barrio2015statistical,ebert2017construction,le2017existence,arjovsky2017wasserstein,solomon2014wasserstein}.
The most popular approach is entropic regularization and application of Sinkhorn's algorithm \cite{cuturi2013sinkhorn}. As it is shown in \cite{gasnikov2015universal,altschuler2017near-linear}, the regularization parameter needs to be chosen small. This can lead to instability of the algorithm. It is a bit better for the accelerated gradient descent \cite{dvurechensky2018computational}, but this method can work slow for small regularization parameter.

We show, how our framework can be used to construct an alternative, which does not require to use Sinkhorn's method with small regularization parameter.\footnote{After we finished our derivations, we found that a close idea was considered in \cite{xie2018fast}. Moreover, as far as we know in practice KL-proximal envelope of Sinkhorn's algorithm was used even earlier (M.~Cutiri, G.~Peyer -- private communication in Les Houches, 2016).}

Optimal transport problem for calculating the  Monge--Kantorovich--Wasserstein distance (MKW-distance) for discrete measures $l,w$ from the standard unit simplex 
is a linear programming (LP) problem

\[
    \sum_{i, j = 1}^{n} c_{ij} x_{ij} 
    \rightarrow 
    \min_{\substack{
        \sum\limits_{j=1}^{n} x_{ij} = l_i, i = 1,...,n; \\ 
        \sum\limits_{i=1}^{n} x_{ij} = w_j, j = 1,...,n; \\ 
        x_{ij} \geq 0, i,j=1,...,n }},
\]
where $\sum\limits_{i=1}^n l_i = \sum\limits_{j=1}^n w_j = 1$. We consider non-accelerated proximal-method with Bregman divergence $V[y](x) = \sum\limits_{i,j = 1}^n x_{ij} \ln(x_{ij} / y_{ij})$ (see \gav{non adaptive variant of Algorithm~\ref{Alg1}} and Example~\ref{prox_ex}). The step of this method reads as
\[
    x^{k+1} = \arg\min\limits_{\substack{
        \sum\limits_{j=1}^{n} x_{ij} = l_i, i = 1,...,n; \\ 
        \sum\limits_{i=1}^{n} x_{ij} = w_j, j = 1,...,n; \\ 
        x_{ij} \geq 0, i,j=1,...,n }} \Bigg\{
            \sum\limits_{i,j=1}^n c_{ij} x_{ij} + \gamma \sum\limits_{i,j=1}^n x_{ij} \ln(x_{ij} / x_{ij}^k)
        \Bigg\},
\]

This \gav{$k$-th} auxiliary minimization problem is exactly the one, which is usually solved by the Sinkhorn's algorithm. The idea of the method is alternating minimization for the dual problem \cite{cuturi2013sinkhorn}.
The complexity of this method is \cite{franklin1989scaling,beck2015convergence,dvurechensky2018computational,stonyakin2019gradient}
\ak{\[
n^2 \widetilde{O}\left(\min\Bigg\{\exp\left(\frac{\bar{c}_k}{\gamma}\right) \left(\frac{\bar{c}_k}{\gamma} + \ln\frac{\bar{c}_k}{\tilde{\e}}\right),\, \frac{\bar{c}_k^2}{\gamma \tilde{\e}}\Bigg\}\right)
\]}
where\footnote{
\gav{By proper rounding of $x^k$ one can guarantee (without loss of generality) that $x^k_{ij} \ge \e / (2 n^2)$ that provide $$\frac{\bar{c}_k}{\gamma} = \frac{\max_{i,j} c_{ij}}{\gamma} + \ln\left(\frac{2 n^2}{\e}\right).$$}
} 
$$\bar{c}_k = \max_{i,j} c_{ij} + \gamma \ln\left(\frac{\max_{i, j} x^k_{i j}}{\min_{i, j} x^k_{i j}}\right)$$
\gav{and} $\tilde{\e}$ is a relative accuracy (\gav{by} function value). When $\gamma$ is small, the complexity is given by the second component and vise versa. At the same time, from Theorem \ref{mainTheoremDL_G} and Example \ref{prox_ex} with inexact model w.r.t. chosen $V[y](x)$ as KL-divergence, it follows that, for any chosen $\gamma$, the number of proximal iterations to obtain accuracy $\e$ is bounded by $\widetilde{O}\left(\gamma/\e\right)$. Thus, we can trade-off the number of outer iterations and the the complexity of inner problem on each iteration by choosing appropriate gamma. It can be shown that for a special choice of $\gamma \gav{= O(\max_k \bar{c}_k)}$, the resulting complexity of the whole method can be \gav{estimated as} $\widetilde{O} \left(\ak{n^4 / \e^2}\right)$ to obtain accuracy
\footnote{Based on the Definition~\ref{solNemirovskiy} and estimate~\eqref{estimate} one can show the following dependence 
$\tilde{\e} = \widetilde{O}(\ak{\e^4 / (\gamma n^4)})$, where $\e$ is a given accuracy (in function value) for initial problem. \gav{To prove this fact one should use relation~\eqref{inexact} with $\|\cdot\| = \|\cdot\|_1$, $R = 2$, $\mu = \gamma$. To bound $L$ we should modify $Q$ (transport polyhedral) by adding constraints: $x_{ij} \ge \e / (4 \ak{n^2})$, $i,j=1,...,n$. Without loss of generality (see Algorithm 2 in \cite{dvurechensky2018computational}) we can consider $l$ and $w$ to be such that $\min_i l_i \ge \e / (2 n)$ and $\min_j w_i \ge \e / (2 n)$. Hence, new polyhedral is well defined and the solution of modified problem is $O(\e)$-solution (by function) of initial problem.
For modified problem one can guarantee that $L = 4 \gamma \ak{n^2} / \e$. According to \eqref{inexact} and Theorem~\ref{GM_inexact} one should solve auxiliary problem with accuracy $\tilde{\e}$ that guarantee $O(\e) = \tilde{\delta} =  (5 \gamma \ak{n^2} / \e) R \sqrt{2 \tilde{\e} / \gamma}$. The only problem is that now we can not directly apply Sinkhorn's algorithm. This problem can be solved by trivial affine transformation of $x$-space. This transformation reduces modified polyhedral to the standard one and we can use Sinkhorn's algorithm. Such a transformation doesn't change (in terms of $O(~)$) the requirements to the accuracy. But one should note, that all these `modifications' aren't necessarily in practice. Since entropy is highly smooth function in positive orthant and zero $x-$components are impossible due to the specificity of Sinkhorn's algorithm we can consider more simple variant of stopping rule for Sinkhorn's method in practice. We do $\bar{N}$ iterations of Sinkhorn's algorithm for inner problem at each outer iteration. Than restart all the procedure from the very beginning with $\bar{N}:=2\bar{N}$, etc. At some moment we detect that further step $\bar{N}:=2\bar{N}$ doesn't change significantly the quality of the solution and we stop here. One can easily show that all these restarts increase the total complexity of the procedure no more than 4 times in comparison with the procedure with (unknown) optimal value of $\bar{N}$.}}
$\e$ in approximation the non-regularized MKW-distance. \gav{In practice this method (Prox Sinkhorn) works significantly better.} \gav{Note, that the best known (for the moment) theoretical bound for} transport problem \gav{is $\widetilde{O} \left(n^2 / \e\right)$} \cite{blanchet2018towards}, whereas Sinkhorn's algorithm has the complexity $\widetilde{O} \left(n^2 / \e^2\right)$.


Figure~\ref{fig:prox-sinkh} shows experimental comparison of Sinkhorn's method and proximal Sinkhorn's method. For the Sinkhorn's method $\gamma$ was chosen in accordance with the theoretical bound $\widetilde{O}\left(\e\right)$. For the proximal Sinkhorn's algorithm, we used the following idea of adaptivity to the parameter $\gamma$. In the first iteration of the proximal method, the problem is solved with overestimated $\gamma$ parameter value. Then we set $\gamma := \gamma/2$ and the problem is solved with the updated value of the parameter, and so on, until a significant increase (for example, 10 times) in the complexity of the auxiliary entropy-linear programming problem in comparison with the initial complexity is detected. The found value of parameter $\gamma$ can be used in next iterations of the proximal method. Also the starting point for the Sinkhorn's method on the next outer iteration can be chosen as the solution of the auxiliary problem from the previous iteration.

In the experiments we use a standard MNIST dataset with images scaled to a size $10\times 10$. The vectors $l$ and $w$ contain the pixel intensities of the first and second images respectively ($n = (width)^2 = 100$). The value of 
$c_{ij}$ 
is equal to the Euclidean distance between the $i$-th pixel from the vector $l$ and the $j$-th pixel from the vector $w$ on the image pixel grid. 

It seems that the described example have different further generalization, e.g. for or Greenkhorn algorithm (instead of Sinkhorn) \cite{lin2019efficient} or can be spread on Wasserstein Barycenter calculation problem \cite{kroshnin2019complexity} .

\begin{figure}[H]
    \centering
    \includegraphics[width=0.6\textwidth]{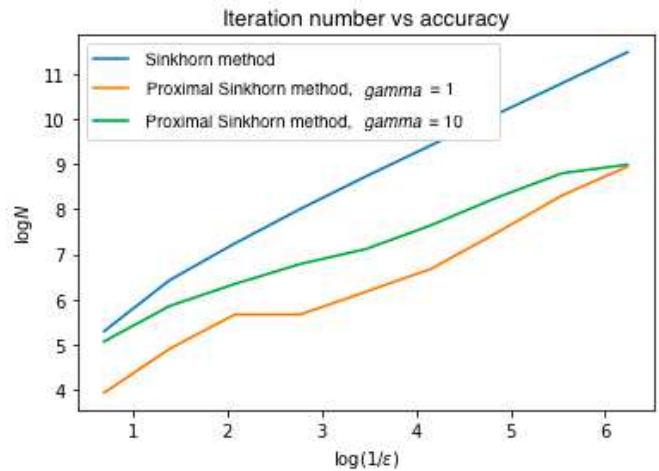}
    \hspace{15pt}
    \vspace{-15pt}
    \smallskip
    \caption{Comparison of \gav{total Sinkhorn's type} iteration number of proximal Sinkhorn method for different values of parameter  $\gamma$ \gav{and number of Sinkhorn's type iterations in Sinkorn's method}.}
    \label{fig:prox-sinkh}
\end{figure}

\end{example}

\begin{example}
{\bf Universal method, \cite{nesterov2015universal}}
\label{universal}

In this example we present a special case of $(\delta, L)$-model which is closely related to universal method (see \cite{nesterov2015universal}). We show that for some choice of $(\delta, L)$-model w.r.t. $\|\cdot\|$ and $\delta_k$ our fast gradient method has the same rate of convergence as accelerated version of the standard universal method.
Let us consider function $f$ is a convex function with Holder continuous (sub)gradient w.r.t. $\|\cdot\|$:
\begin{equation*}
\norm{\nabla f(x) - \nabla f(y)}_* \leq L_\nu\norm{x - y}^\nu\,\,\,\forall x,y \in Q.
\end{equation*}

For functions with Holder continuous (sub)gradient we can write the following inequality (\cite{nesterov2015universal}):
\begin{gather}
0 \leq f(x) - f(y) - \langle\nabla f(y), x - y \rangle \leq \frac{L(\delta)}{2}\norm{x - y}^2 + \delta ,\,\,\, \forall x,y \in Q,
\end{gather}
where \begin{gather*}L(\delta)=L_\nu\left[\frac{L_\nu}{2\delta}\right]^\frac{1-\nu}{1+\nu}\end{gather*} and $\delta > 0$ is a free parameter.



From the last inequality one can see that we can take $\psi_{\delta_k}(x,y)=\langle\nabla f(y), x - y \rangle$ and $f_{\delta_k}(y)=f(y)$.

Let us take \begin{equation}
\label{delta_k_univ}
    \delta_k=\e \frac{\alpha_{k+1}}{4A_{k+1}} \,\,\,\forall k.
\end{equation} where $\e$ is the required accuracy of the solution by function. From theorem \ref{mainTheoremDL}  with our assumptions we have the following convergence rate:
\begin{equation}
	f(x_N) - f(x_*) \leq \frac{R^2}{A_N} + \frac{\e}{2}.
\end{equation} 

As in \cite{nesterov2015universal} we can show that 
\begin{equation}
\label{AN_universal}
A_N \geq \frac{N^\frac{1+3\nu}{1+\nu}\epsilon^\frac{1-\nu}{1+\nu}}{2^\frac{2+4\nu}{1+\nu}L_\nu^\frac{2}{1+\nu}}.\end{equation}
Using \eqref{AN_universal} we can show the following upper bound for the number of steps for getting $\epsilon$-solution:
$$N \leq \inf_{\nu\in[0,1]}\left[2^\frac{3+5\nu}{1+3\nu}\left(\frac{L_\nu R^{1+\nu}}{\e}\right)^\frac{2}{1+3\nu}\right].$$
This estimate is optimal (see \cite{guzman2015lower}).
\end{example}

\begin{example}
{\bf Universal conditional gradient (Frank--Wolfe) method}
\label{example_frank}

Let us consider convex problem~\eqref{Problem}, where $f$ has Holder continuous (sub)gradient w.r.t. $\|\cdot\|$. Assume that $V[y](x) \leq R_Q^2$ for all $x,y \in Q$. Sometimes in practice auxiliary problem \eqref{equmir2DL} can be hard (\cite{ben-tal2015lectures, nesterov2018complexity}). In\footnote{For details see also \cite{bubeck2015convex,ben-tal2015lectures, harchaoui2015conditional,anikin2015efficient,nesterov2018complexity}.} \cite{jaggi2013revisiting} it was shown that conditional gradient method (Frank--Wolfe) can be useful for some of these problems.
In algorithms \ref{Alg1} and \ref{Alg2} from sections \ref{GM} and \ref{fastGradMethod} we have auxiliary optimization problems \eqref{equmir2DL_G} and \eqref{equmir2DL}.
Instead of functions in auxiliary optimization problems \eqref{equmir2DL_G} and \eqref{equmir2DL} let us take 
\begin{gather*}\widetilde{\phi}_{k+1}(x) = \at2{L_{k+1}\alpha_{k+1}}\psi_{\delta_k}(x, x_{k})\end{gather*}
and
\begin{gather*}\widetilde{\phi}_{k+1}(x) = \at2{L_{k+1}}\alpha_{k+1}\psi_{\delta_k}(x, y_{k+1}),\end{gather*}
respectively.
With this substitution our method from section \ref{GM} becomes Frank--Wolfe method. Further we show that Frank--Wolfe is a special case of methods from sections \ref{GM} and \ref{fastGradMethod}. Moreover, we provide universal Frank--Wolfe method combining ideas from Frank--Wolfe method and universal method \cite{nesterov2015universal}.
Let us look at this substitution from the view of an error $\widetilde{\delta}_k$ where $\widetilde{\delta}_k$ is an error in terms of definition \eqref{solNemirovskiy}. We can show that it is enough to take
$\widetilde{\delta}_k = 2\at2{L_{k+1}}R^2_Q$ for all $k \geq 0$, where $R_Q$ is a diameter of a set $Q$. Also let us take \begin{gather*}\delta_k = \e \frac{\alpha_{k+1}}{4A_{k+1}} \,\forall k,\end{gather*} where $\e$ is the accuracy of
the solution by function.
It is enough to do
\begin{gather}
\label{frank_complexity}
N \leq \inf_{\nu\in(0,1]}\left[2^\frac{5+7\nu}{2\nu}\left(\frac{L_\nu R_Q^{1+\nu}}{\e}\right)^\frac{1}{\nu}\right].
\end{gather}
steps in order to find an $\e$-solution of the optimization problem. Constants $L_\nu$ and $\nu$ are defined in example \ref{universal}. Let us prove it.
Let us first show that it enough to take $\widetilde{\delta}_k = 2\at2{L_{k+1}}R^2_Q$ for all $k \geq 0$:
\begin{gather*}
\exists h \in \partial\phi_{k+1}(u_{k+1}), \exists g \in \partial\widetilde{\phi}_{k+1}(u_{k+1}),\,\,\, \langle h, x - u_{k+1}  \rangle =\\= \langle g, x - u_{k+1}  \rangle + \at2{L_{k+1}}\langle \nabla_{u_{k+1}} V[u_k](u_{k+1}), x - u_{k+1}\rangle \geq \at2{L_{k+1}}\langle \nabla_{u_{k+1}} V[u_k](u_{k+1}), x - u_{k+1}\rangle = \\ =
-\at2{L_{k+1}}V[u_{k}](u_{k+1}) - \at2{L_{k+1}}V[u_{k+1}](x) + \at2{L_{k+1}}V[u_{k}](x) \geq -2\at2{L_{k+1}}R^2_Q.
\end{gather*}
Thus the point $u_{k+1}$ is a $\widetilde{\delta}_k$-solution in sense of Definition \ref{solNemirovskiy}.

It is left to proof inequality \eqref{frank_complexity}. Using Theorem \ref{mainTheoremDLNew} we can show:
\begin{equation*}
f(x_N) - f(x_*) \leq \frac{R^2}{A_N} + \frac{\e}{2} + \frac{2R^2_QN}{A_N} \leq \frac{3R^2_QN}{A_N} + \frac{\e}{2}.
\end{equation*}

Using \eqref{AN_universal} we can show the following upper bound for the number of steps for getting $\epsilon$-solution:

\begin{gather*}
N \leq \inf_{\nu\in(0,1]}\left[2^\frac{5+7\nu}{2\nu}\left(\frac{L_\nu R_Q^{1+\nu}}{\e}\right)^\frac{1}{\nu}\right].
\end{gather*}

\end{example}

\section{Proof of Theorem \ref{mainTheoremDL_G}}\label{ProofGM}

Let us propose generalization of theorem \ref{mainTheoremDL} where we take in account inaccuracies arise from the inexact solution of auxiliary problems. The first sequence $\{ \delta_k \}_{k=0}^{N-1}$ is a sequence such that for any $k$ there is a $(\delta_k, L)$-model for $f$ (w.r.t. $V[y](x)$ and w.r.t. $\|\cdot\|$ in Appendix~\ref{profFGM}). 
Numbers in the second sequence
$\{\widetilde{\delta}_k\}_{k=0}^{N-1}$ are the accuracies of the solution of the auxiliary problem in terms of Definition \ref{solNemirovskiy}. 
\begin{theorem}\label{GM_inexact}
	Let $V[x_0](x_*) \leq R^2$, where $x_0$ is the starting point, and $x_*$ is the closest point of the minimum to the point $x_0$ in the sense of Bregman divergence, and \begin{gather*}\bar{x}_N= \frac{1}{A_N}\sum_{k=0}^{N-1}\alpha_{k+1} x_{k+1}.
	\end{gather*} For the proposed algorithm we have the following convergence rate:
	\begin{align}
	f(\bar{x}_N) - f(x_*) &\leq \frac{R^2}{A_N} + \frac{1}{A_N}\sum_{k=0}^{N-1}\at2{\alpha_{k+1}}\widetilde{\delta}_k + \frac{2}{A_N}\sum_{k=0}^{N-1}\alpha_{k+1}\delta_k \\
	&\leq \frac{2LR^2}{N} + \at2{\frac{1}{A_N}}\sum_{k=0}^{N-1}\at2{\alpha_{k+1}}\widetilde{\delta}_k + \frac{2}{A_N}\sum_{k=0}^{N-1}\alpha_{k+1}\delta_k.
	\end{align}
\end{theorem}

The full proof of this theorem includes two lemmas. Let us formulate and prove lemmas.
\begin{lemma}
	Let $\psi(x)$ be a convex function and
	\begin{gather*}
	y = {\arg\min_{x \in Q}}^{\widetilde{\delta}} \{\psi(x) + \at2{\beta} V[z](x)\},
	\end{gather*}
	\at2{where $\beta \geq 0$.}
Then
	\begin{equation*}
	\psi(x) + \at2{\beta}V[z](x) \geq \psi(y) + \at2{\beta}V[z](y) + \at2{\beta}V[y](x) - \widetilde{\delta} ,\,\,\, \forall x \in Q.
	\end{equation*}
	\label{lemma_maxmin_2}
\end{lemma}

\begin{proof}
	By definition \ref{solNemirovskiy}:
	\begin{gather*}
		\exists g \in \partial\psi(y), \,\,\, \langle g + \at2{\beta}\nabla_y V[z](y), x - y \rangle \geq -\widetilde{\delta} ,\,\,\, \forall x \in Q.
	\end{gather*}
	Then inequality
	\begin{gather*}
		\psi(x) - \psi(y) \geq \langle g, x - y \rangle \geq \langle \at2{\beta}\nabla_y V[z](y), y - x \rangle - \widetilde{\delta}
	\end{gather*}
and equality
	\begin{gather*}
	\langle \nabla_y V[z](y), y - x \rangle=\langle \nabla d(y) - \nabla d(z), y - x \rangle=d(y) - d(z) - \langle \nabla d(z), y - z \rangle +\\ + d(x) - d(y) - \langle \nabla d(y), x - y \rangle - d(x) + d(z) + \langle \nabla d(z), x - z \rangle=\\=
	V[z](y) + V[y](x) - V[z](x)
	\end{gather*}
complete the proof.
\end{proof}

\begin{lemma}
\label{lemma_maxmin_3DL_G}
	$\forall x \in Q$ we have
	\begin{gather*}
		\alpha_{k+1}f(x_{k+1}) - \alpha_{k+1}f(x)\leq V[x_k](x) - V[x_{k+1}](x) + \at2{\alpha_{k+1}}\widetilde{\delta}_k + 2\alpha_{k+1}\delta_k.
	\end{gather*}
\end{lemma}

\begin{proof}
	\begin{gather*}
	f(x_{k+1}) \leqarg{\eqref{exitLDL_G}, \eqref{model_def}} f_{\delta_k}(x_{k}) + \psi_{\delta_k}(x_{k+1}, x_{k}) + L_{k+1} V[x_{k}](x_{k+1}) + 2\delta_k = \\ =
	f_{\delta_k}(x_{k}) + \psi_{\delta_k}(x_{k+1}, x_{k}) + \frac{1}{\alpha_{k+1}}V[x_{k}](x_{k+1}) + 2\delta_k
	\leqarg{{\tiny \circled{1}}} \\\leq
	 f_{\delta_k}(x_{k}) + \psi_{\delta_k}(x,x_{k})
	 	 + \frac{1}{\alpha_{k+1}}V[x_k](x) - \frac{1}{\alpha_{k+1}}V[x_{k+1}](x) + \at2{\widetilde{\delta}_k} + 2\delta_k \leqarg{\eqref{model_def}} \\ \leq f(x) + \frac{1}{\alpha_{k+1}}V[x_k](x) - \frac{1}{\alpha_{k+1}}V[x_{k+1}](x) + \at2{\widetilde{\delta}_k} + 2\delta_k
	\end{gather*}
	{\small \circled{1}}~--- from lemma \ref{lemma_maxmin_2} with
	$\psi(x)=\psi_{\delta_k}(x, x_{k})$ \at2{and $\beta=1 / \alpha_{k+1}$.}  
\end{proof}

\begin{remark}
\leavevmode
\label{remark_maxmin}
Let us show that
$L_{k} \leq 2L\quad \forall k \geq 0$.
For $k=0$ this is true from the fact that $L_0 \leq L$. For $k \geq 1$ this follows from the fact that we leave the inner cycle
earlier than $L_{k}$ will be greater than $2L$. The exit from the cycle is guaranteed by the condition that there is an $(\delta_k, L)$-model for $f(x)$ at any point $x \in Q$.
\end{remark}

We are ready to proof the theorem.

\begin{proof}

Let us sum up the inequality from Lemma \ref{lemma_maxmin_3DL_G} at $k=0, ..., N - 1$
	\begin{gather*}
		\sum_{k=0}^{N-1}\alpha_{k+1}f(x_{k+1}) - A_N f(x) \leq V[x_0](x) - V[x_N](x) + \sum_{k=0}^{N-1}\at2{\alpha_{k+1}}\widetilde{\delta}_k + 2\sum_{k=0}^{N-1}\alpha_{k+1}\delta_k.
	\end{gather*}
	With $x=x_*$ we have that
	\begin{gather*}
		\sum_{k=0}^{N-1}\alpha_{k+1}f(x_{k+1}) - A_N f(x_*) \leq R^2 - V[x_N](x_*) + \sum_{k=0}^{N-1}\at2{\alpha_{k+1}}\widetilde{\delta}_k + 2\sum_{k=0}^{N-1}\alpha_{k+1}\delta_k.
	\end{gather*}
	Since $V[x_N](x_*) \geq 0$ we obtain inequality
	\begin{gather*}
		\sum_{k=0}^{N-1}\alpha_{k+1}f(x_{k+1}) - A_N f(x_*) \leq R^2 + \sum_{k=0}^{N-1}\at2{\alpha_{k+1}}\widetilde{\delta}_k + 2\sum_{k=0}^{N-1}\alpha_{k+1}\delta_k.
	\end{gather*}
	Let us divide both parts by $A_N$.
	\begin{gather*}
	\frac{1}{A_N}\sum_{k=0}^{N-1}\alpha_{k+1}f(x_{k+1}) - f(x_*) \leq \frac{R^2}{A_N} + \frac{1}{A_N}\sum_{k=0}^{N-1}\at2{\alpha_{k+1}}\widetilde{\delta}_k + \frac{2}{A_N}\sum_{k=0}^{N-1}\alpha_{k+1}\delta_k.
	\end{gather*}
	Using the convexity of $f(x)$ we can show that
	\begin{gather*}
		f(\bar{x}_N) - f(x_*) \leq \frac{R^2}{A_N} + \frac{1}{A_N}\sum_{k=0}^{N-1}\at2{\alpha_{k+1}}\widetilde{\delta}_k + \frac{2}{A_N}\sum_{k=0}^{N-1}\alpha_{k+1}\delta_k.
	\end{gather*}
	Remains only to prove that $$\frac{1}{A_N} \leq \frac{2L}{N}.$$ 
	As it follows from definition~\ref{model} and remark \ref{remark_maxmin} for all $k \geq 0$
$L_{k} \leq 2L$.
	Thus,  we have that $$\alpha_{k} = \frac{1}{L_k} \geq \frac{1}{2L}$$ and $$A_N = \sum_{k=0}^{N}\alpha_k \geq \frac{N}{2L}.$$

The estimate of the total number of oracle calls is estimated in the same way as in \cite{nesterov2006cubic}.


\end{proof}

\section{Proof of Theorem \ref{mainTheoremDL}} \label{profFGM}




\begin{theorem}\label{mainTheoremDLNew}
	Let $V[x_0](x_*) \leq R^2$, where $x_0$ is the starting point and $x_*$ is the nearest minimum point to $x_0$ in the sense of Bregman divergence. For the proposed algorithm the following inequality holds:
	\begin{align*}
			f(x_N) - f_* &\leq \frac{R^2}{A_{N}} + \frac{2\sum_{k=0}^{N-1}A_{k+1}\delta_k}{A_{N}} + \frac{\sum_{k=0}^{N-1}\at2{\frac{\widetilde{\delta}_k}{L_{k+1}}}}{A_{N}} \\
			&\leq \frac{8LR^2}{(N+1)^2} + \frac{2\sum_{k=0}^{N-1}A_{k+1}\delta_k}{A_{N}} + \frac{\sum_{k=0}^{N-1}\at2{\frac{\widetilde{\delta}_k}{L_{k+1}}}}{A_{N}}.
	\end{align*}
\end{theorem}

Let us proof auxiliary lemmas.
\begin{lemma}
	\label{lemma_maxmin_1}
	Suppose that for sequence $\alpha_k$ it is satisfied
	\begin{align*}
	\alpha_0=0,\,\,\,
	A_k=\sum_{i=0}^{k}\alpha_i,\,\,\,
	A_k=L_{k}\alpha_k^2,\,\,\,
	\end{align*}
	where $L_k \leq 2L\,\forall k\geq0$ (see Remark \ref{remark_maxmin}).
	Then $\forall k \geq 1$ the following inequality holds:
	 \begin{align}
	 \label{lemma_maxmin_1_1}
	 A_k \geq \frac{(k+1)^2}{8L}.
	 \end{align}
\end{lemma}

\begin{proof}
	Let $k=1$.
	\begin{equation*}
	\alpha_1=L_{1}\alpha_1^2
	\end{equation*}
	and
	\begin{equation*}
	A_1=\alpha_1=\frac{1}{L_1} \geq \frac{1}{2L}.
	\end{equation*}
	Let $k \geq 2$, then
	\begin{equation*}
	L_{k+1}\alpha^2_{k+1}=A_{k+1}\Leftrightarrow
	\end{equation*}
	\begin{equation*}
	L_{k+1}\alpha^2_{k+1}=A_{k} + \alpha_{k+1}\Leftrightarrow
	\end{equation*}
	\begin{equation*}
	L_{k+1}\alpha^2_{k+1} - \alpha_{k+1} - A_{k}=0.
	\end{equation*}
	Solving this quadratic equation we will take the largest root, therefore
	\begin{equation*}
	\alpha_{k+1}=\frac{1 + \sqrt{\uprule 1 + 4L_{k+1}A_{k}}}{2L_{k+1}}.
	\end{equation*}
	By induction, let the inequality \eqref{lemma_maxmin_1_1} be true for $k$, then:
	\begin{gather*}
	\alpha_{k+1}=\frac{1}{2L_{k+1}} + \sqrt{\frac{1}{4L_{k+1}^2} + \frac{A_{k}}{L_{k+1}}} \geq
	\frac{1}{2L_{k+1}} + \sqrt{\frac{A_{k}}{L_{k+1}}} \geq \\\geq
	\frac{1}{4L} + \frac{1}{\sqrt{2L}}\frac{k+1}{2\sqrt{2L}} =
	\frac{k+2}{4L}
	\end{gather*}
The last inequality follows from the induction hypothesis. Finally, we obtain, that
	\begin{equation*}
	\alpha_{k+1} \geq \frac{k+2}{4L}
	\end{equation*}
and
	\begin{equation*}
	A_{k+1}=A_k + \alpha_{k+1}=\frac{(k+1)^2}{8L} + \frac{k+2}{4L} \geq \frac{(k+2)^2}{8L}.
	\end{equation*}
\end{proof}

\begin{lemma}
\label{lemma_maxmin_3DL}
	For each $x \in Q$ we have:
	\begin{equation*}
		A_{k+1} f(x_{k+1}) - A_{k} f(x_{k}) + V[u_{k+1}](x) - V[u_{k}](x) \leq \alpha_{k+1} f(x) + 2\delta_k A_{k+1} + \at2{\frac{\widetilde{\delta}_k}{L_{k+1}}}.
	\end{equation*}
\end{lemma}

\begin{proof}
	\begin{gather*}
	f(x_{k+1}) \leqarg{\eqref{exitLDL}, \eqref{model_def}} f_{\delta_k}(y_{k+1}) + \psi_{\delta_k}(x_{k+1},y_{k+1}) + \frac{L_{k+1}}{2}\norm{x_{k+1} - y_{k+1}}^2 + 2\delta_k \eqarg{\eqref{eqxmir2DL}} \\=
	f_{\delta_k}(y_{k+1}) + \psi_{\delta_k}\left(\frac{\alpha_{k+1}u_{k+1} + A_k x_k}{A_{k+1}},y_{k+1}\right) +\\+ \frac{L_{k+1}}{2}\norm{\frac{\alpha_{k+1}u_{k+1} + A_k x_k}{A_{k+1}} - y_{k+1}}^2 + 2\delta_k \leqarg{conv-ty, \eqref{eqymir2DL}} \\
	\leq f_{\delta_k}(y_{k+1}) +
	\frac{\alpha_{k+1}}{A_{k+1}}\psi_{\delta_k}(u_{k+1}, y_{k+1}) + \\+
	 \frac{A_k}{A_{k+1}}\psi_{\delta_k}(x_k, y_{k+1}) + \frac{L_{k+1} \alpha^2_{k+1}}{2 A^2_{k+1}}\norm{u_{k+1} - u_k}^2 + 2\delta_k= \\=
	 \frac{A_k}{A_{k+1}}(f_{\delta_k}(y_{k+1}) + \psi_{\delta_k}(x_k, y_{k+1}))
	 + \\+
	 \frac{\alpha_{k+1}}{A_{k+1}}(f_{\delta_k}(y_{k+1}) +
	 \psi_{\delta_k}(u_{k+1}, y_{k+1}))+ \\
	 + \frac{L_{k+1} \alpha^2_{k+1}}{2 A^2_{k+1}}\norm{u_{k+1} - u_k}^2 + 2\delta_k\eqarg{{\tiny \circled{1}}} \\ =
	 \frac{A_k}{A_{k+1}}(f_{\delta_k}(y_{k+1}) + \psi_{\delta_k}(x_k,y_{k+1}))
	 + \\+
	 \frac{\alpha_{k+1}}{A_{k+1}}(f_{\delta_k}(y_{k+1}) + \psi_{\delta_k}(u_{k+1},y_{k+1})
	 + \frac{1}{2 \alpha_{k+1}}\norm{u_{k+1} - u_k}^2) + 2\delta_k\leq \\ \leq
	 \frac{A_k}{A_{k+1}}(f_{\delta_k}(y_{k+1}) + \psi_{\delta_k}(x_k,y_{k+1}))
	 + \\+
	 \frac{\alpha_{k+1}}{A_{k+1}}(f_{\delta_k}(y_{k+1}) + \psi_{\delta_k}(u_{k+1},y_{k+1})
	 + \frac{1}{\alpha_{k+1}}V[u_k](u_{k+1})) + 2\delta_k\leqarg{\tiny \circled{2}} \\ \leq
	 \frac{A_k}{A_{k+1}} f(x_k) +	 \frac{\alpha_{k+1}}{A_{k+1}}\Big(f_{\delta_k}(y_{k+1}) + \psi_{\delta_k}(x,y_{k+1}) + \\
	 + \frac{1}{\alpha_{k+1}}V[u_k](x) - \frac{1}{\alpha_{k+1}}V[u_{k+1}](x) + \at2{\frac{\widetilde{\delta}_k}{\alpha_{k+1}L_{k+1}}}\Big) + 2\delta_k \leqarg{\eqref{model_def}}\\\leq
	 \frac{A_k}{A_{k+1}} f(x_k) +
	 \frac{\alpha_{k+1}}{A_{k+1}} f(x)
	 + \frac{1}{A_{k+1}}V[u_k](x) - \frac{1}{A_{k+1}}V[u_{k+1}](x) + 2\delta_k + \at2{\frac{\widetilde{\delta}_k}{A_{k+1}L_{k+1}}}.
	\end{gather*}
	
	{\small \circled{1}}~--- from $A_k=L_{k}\alpha^2_k$.
	
	{\small \circled{2}}~--- from lemma \ref{lemma_maxmin_2} and \eqref{model_def}.
\end{proof}

We are ready to proof the theorem.

\begin{proof}

Let us sum up the inequality of lemma \ref{lemma_maxmin_3DL} 
\gav{for} $k=0, ..., N - 1$
	\begin{gather*}
		A_{N} f(x_N) - A_{0} f(x_0) + V[u_N](x) - V[u_0](x) \leq\\\leq (A_N - A_0)f(x) + 2\sum_{k=0}^{N-1}A_{k+1}\delta_k  + \sum_{k=0}^{N-1}\at2{\frac{\widetilde{\delta}_k}{L_{k+1}}}
	\end{gather*}
	and
	\begin{gather*}
		A_{N} f(x_N) + V[u_N](x) - V[u_0](x) \leq A_N f(x) + 2\sum_{k=0}^{N-1}A_{k+1}\delta_k  + \sum_{k=0}^{N-1}\at2{\frac{\widetilde{\delta}_k}{L_{k+1}}}.
	\end{gather*}
	Let us take $x=x_*$:
	\begin{gather*}
		A_{N} (f(x_N) - f_*) \leq R^2 + 2\sum_{k=0}^{N-1}A_{k+1}\delta_k + \sum_{k=0}^{N-1}\at2{\frac{\widetilde{\delta}_k}{L_{k+1}}}.
	\end{gather*}
	We divide both sides of the inequality by $A_N$ and finally we get, that
	\begin{gather*}	
		f(x_N) - f_* \leq \frac{R^2}{A_{N}} + \frac{2\sum_{k=0}^{N-1}A_{k+1}\delta_k}{A_{N}} + \frac{\sum_{k=0}^{N-1}\at2{\frac{\widetilde{\delta}_k}{L_{k+1}}}}{A_{N}} \leqarg{\tiny \circled{1}}\\\leq \frac{8LR^2}{(N+1)^2} + \frac{2\sum_{k=0}^{N-1}A_{k+1}\delta_k}{A_{N}} + \frac{\sum_{k=0}^{N-1}\at2{\frac{\widetilde{\delta}_k}{L_{k+1}}}}{A_{N}}.
	\end{gather*}
		
	{\small \circled{1}}~--- from lemma \ref{lemma_maxmin_1}.
\end{proof}

\section{The Case of Strongly Convex Objective}\label{SC}

Now we consider the case of a strongly convex objective. The following assumption allows us to prove a lin aerrate of convergence for Algorithm \ref{Alg1}.

\begin{definition}\label{defRelStronglyConvex}
 Say that the function $f$ is a right relative $\mu$-strongly convex if the following inequality
    \begin{gather*}
      \mu V[y](x) \leq f(x) - f(y)- \psi_{\delta}(x, y).
    \end{gather*}
holds.
\end{definition}

Recall that for a strongly convex in the usual sense of the functional $f$ the following inequality will be true
\begin{gather*}
  \frac {\mu}{2}||x-y||^2 \leq f(x) - f(y)- \psi_{\delta}(x, y).
\end{gather*}

\begin{remark}\label{RemEuclidBreg}
Let us remind that if $d(x-y) \leq C_n\norm{x-y}^2$ for $C_n = O(\log n)$ (where n is dimension of vectors from $Q$), then $V[y](x) \leq C_n\norm{x-y}^2$. This assumption is true for many standard proximal setups. In this case the condition of $(\mu C_n)$-strong convexity
$$\mu C_n \norm{x-y}^2 + f_\delta(y) + \psi_\delta(x,y) \leqslant f(x)$$
entails right relative strong convexity:
$$\mu V[y](x) + f_\delta(y) + \psi_\delta(x,y) \leqslant f(x).$$
\end{remark}

\noindent

After $k$ iterations of non-adaptive version of Algorithm \ref{Alg1} with a constant step $\alpha_{i} = \frac{1}{L}$\\ ($i = 1,...,k$), using lemma \ref{lemma_maxmin_2},  we have:

$$ -\widetilde{\delta} \leq \psi_\delta(x,x_k) - \psi_\delta(x_{k+1},x_k) + LV[x_k](x)-LV[x_{k+1}](x)-LV[x_k](x_{k+1}),$$
	therefore,
	\begin{equation}
	\label{eq1}
		LV[x_{k+1}](x)\leq\widetilde{\delta} + \psi_\delta(x,x_k)-\psi_\delta(x_{k+1}, x_k) + LV[x_k](x)-LV[x_k](x_{k+1}).
	\end{equation}
Further, $\psi_{\delta}(x, y)$ is a ($\delta$, L)-model w.r.t. $V[y](x)$ and from
	$$ f(x_{k+1}) \leq f_\delta(x_k)+\psi_\delta(x_{k+1},x_k)+LV[x_k](x_{k+1}) + \delta,$$
	we get
	$$-LV[x_k](x_{k+1}) \leq \delta - f(x_{k+1}) + f_\delta(x_k) + \psi_\delta(x_{k+1}, x_k) .$$
	 
Now \eqref{eq1} means
	\begin{equation}
		\label{eq2}
		LV[x_{k+1}](x) \leq \widetilde{\delta} + \delta - f(x_{k+1}) + f_\delta(x_k) + \psi_\delta(x, x_k) + LV[x_k](x).
	\end{equation}
	
	Using right relative strong convexity, we have:
	$$f(x) \geq f_\delta(x_k) + \psi_\delta(x,x_k) + \mu V[x_k](x) $$
	or
	$$f_\delta(x_k) + \psi_\delta(x,x_k) \leq f(x) - \mu V[x_k](x).$$
	Considering \eqref{eq2}, we obtain:
	\begin{equation}
		\label{eq3}
		LV[x_{k+1}](x) \leq \widetilde{\delta} + \delta + f(x) - f(x_{k+1})+(L-\mu)V[x_k](x).
	\end{equation}
	For $x = x_*$ we have:
	\begin{gather*}
		V[x_{k+1}](x_*) \leq \left(f(x_*) - f(x_{k+1})+\delta+\widetilde{\delta}\right)\dfrac{1}{L}+\left(1-\dfrac{\mu}{L}\right)V[x_k](x_*) \leq\\
		\leq \left(f(x_*) - f(x_{k+1})+\delta+\widetilde{\delta}\right)\dfrac{1}{L} + \\ + \left(1-\dfrac{\mu}{L}\right)\left(\left(f(x_*) - f(x_{k})+\delta+\widetilde{\delta}\right)\dfrac{1}{L} + \left(1-\dfrac{\mu}{L}\right)V[x_{k-1}](x_*)\right) \leq\\
		\le \ldots \le \left(1-\dfrac{\mu}{L}\right)^{k+1} V[x_0](x_*)+ \dfrac{1}{L}\sum\limits_{i=0}^k \left(1-\dfrac{\mu}{L}\right)^i\left(f(x_*) - f(x_{k+1-i})+\delta+\widetilde{\delta}\right).
	\end{gather*}
	Therefore, we have
	$$\dfrac{1}{L}\sum\limits_{i=0}^k\left(1-\dfrac{\mu}{L}\right)^i(f(x_{k+1-i})-f(x_*)) \leq \left(1-\dfrac{\mu}{L}\right)^{k+1}V[x_0](x_*)+\dfrac{1}{L}\sum\limits_{i=0}^k\left(1-\dfrac{\mu}{L}\right)^i(\delta+\widetilde{\delta}).$$
	
	Let $y_k = \argmin_{i = 1,..., k}(f(x_i))$. Then using this definition and the fact that $$\frac{1}{L}\sum\limits_{i=0}^k\left(1-\frac{\mu}{L}\right)^i =  \frac{1}{\mu}\left(1-\left(1 - \frac{\mu}{L}\right)^{k+1}\right),$$
	we obtain
	$$f(y_{k+1}) - f(x_*) \le \mu \dfrac{\left(1 - \frac{\mu}{L}\right)^{k+1}}{1-\left(1-\frac{\mu}{L}\right)^{k+1}}V[x_0](x_*) + \delta + \widetilde{\delta} \leqslant L(1-\frac{\mu}{L})^{k+1}V[x_0](x_*) + \delta + \widetilde{\delta},$$
	and, using the fact that $ e^{-x}\ge 1-x~~ \forall x \ge 0$, we conclude that	

		\begin{equation}\label{agaf_eq4}
		f(y_{k+1}) - f(x_*)\le LV[x_0](x_*) \exp\left((-k+1)\dfrac{\mu}{L}\right) + \delta + \widetilde{\delta}.
		\end{equation}

Let $x = x_*$ in \eqref{eq3}, from which $f(x_*) \leq f(x_{k+1})$ and
$$LV[x_{k+1}](x_*) \leq \widetilde{\delta} + \delta + (L-\mu)V[x_k](x_*),$$
i.e.
$$V[x_{k+1}](x_*) \leq \dfrac{1}{L}(\delta + \widetilde{\delta}) + \left(1 - \dfrac{\mu}{L}\right)V[x_k](x_*).$$
Further,
		\begin{align*}
	V[x_{k+1}](x_*) \leq \dfrac{1}{L}(\delta+\widetilde{\delta}) + \left( 1-\dfrac{\mu}{L}\right)\left(\dfrac{1}{L}(\delta + \widetilde{\delta}) + \left(1 - \dfrac{\mu}{L}\right)V[x_{k-1}](x_*)\right) \leq \ldots \leq \\
	\leq \dfrac{1}{L}(\widetilde{\delta} + \delta)\left(1+ \left(1-\dfrac{\mu}{L}\right)+ \ldots + \left(1-\dfrac{\mu}{L}\right)^k \right) + \left(1-\dfrac{\mu}{L}\right)^{k+1}V[x_0](x_*).
	\end{align*}
	
	Therefore, taking into account the following fact
	$\sum\limits_{i=0}^k\left(1-\frac{\mu}{L}\right)^i < \frac{1}{1 - \left(1 - \frac{\mu}{L}\right)} = \frac{L}{\mu}$, we obtain

		\begin{equation}\label{agaf_eq5}
			V[x_{k+1}](x_*) \leq \dfrac{1}{\mu}(\delta+\widetilde{\delta}) + \left(1-\dfrac{\mu}{L}\right)^{k+1}V[x_0](x_*).
		\end{equation}

Thus, we have the following result

\begin{theorem}
Assume that function $f$ is a right relatively strongly convex and $\psi_{\delta}(x, y)$ is a ($\delta$, L)-model w.r.t. $V[y](x)$. Then, after of $k$ iterations of non-adaptive version of Algorithm \ref{Alg1}, $f$ satisfies \eqref{agaf_eq4} and \eqref{agaf_eq5}.
\end{theorem}

In other words, if function satisfies right relative strong convexity and relative smoothness, then after performing $O(\log(\frac{1}{\varepsilon}))$ iterations we can achive an accuracy of $\varepsilon$ accurate to term $O(\delta+\widetilde{\delta})$.

Let us consider the case of a strongly convex functional $f$ and show how to accelerate the work of Algorithms \ref{Alg1} and \ref{Alg2} using the restart technique. Let us assume that 
\begin{equation}
\psi_{\delta}(x,x_*) \ge 0\,\,\, \forall x \in Q.
\end{equation}
Note the this assumption is natural, e.g. $\psi_{\delta}(x, y):= \langle\nabla f(y), x - y \rangle \,\,\, \forall x,y \in Q$. We also modify the concept of relative $\mu$-strongly convexity in the following way

\begin{definition}\label{defRelStronglyConvexRest}
 Say that the function $f$ is a left relative $\mu$-strongly convex if the following inequality
    \begin{gather*}
      \mu V[x](y) \leq f(x) - f(y)- \psi_{\delta}(x, y).
    \end{gather*}
holds.
\end{definition}

Note that concepts of right and left relative strongly convexity from Definitions \ref{defRelStronglyConvex} and \ref{defRelStronglyConvexRest} are equivalent in the case of assumption from Remark \ref{RemEuclidBreg} ($V[x](y) \leq C_n \|x - y\|^2$ for each $x, y \in Q$).

\noindent
  \begin{theorem}\label{Th11Section4}
    Let $f$ be a left relative $\mu$-strongly convex function and $\psi_{\delta}(x,y)$ is a ($\delta$, L)-model w.r.t. $V[y](x)$. Then, using the restarts of Algorithm \ref{Alg1}, we obtain the estimate
\begin{gather}\label{estimateFromTh5.1}
      V[\bar{x}_{N_p}]({x_*}) \leq \varepsilon + \frac {2\tilde{\delta}}{\mu}+ \frac {4\delta}{\mu}
    \end{gather}
for a given $\varepsilon>0$. The total number for iterations of Algorithm \ref{Alg1} not exceeding
      \begin{equation}\label{iterations}
        N = \left\lceil {\log_2 {\frac{R^2}{\varepsilon}}} \right\rceil \cdot \left\lceil {\frac {4L}{\mu}} \right\rceil.
     \end{equation}
  \end{theorem}

\begin{proof}	
By Definition \ref{defRelStronglyConvexRest} and Theorem \ref{GM_inexact} we have
\begin{gather*}
    \mu V[\bar{x}_{N_{1}}](x_*) \leq f(\bar{x}_{N_{1}})-f(x_*) \leq \frac {2LV[x_0](x_*)}{N_1}+\widetilde{\delta}+2\delta.
 \end{gather*}
Further, due to the following inequality
  \begin{equation}\label{eqD}
    V[\bar{x}_{N_{1}}](x_*) \leq \frac {2LV[{x_0}](x_*)}{\mu N_1}+\frac {\widetilde{\delta}}{\mu}+ \frac {2\delta}{\mu}
  \end{equation}
let's choose the smallest number of steps $N_1$:
  \begin{gather*}
    V[\bar{x}_{N_{1}}]({x_*}) \leq \frac {1}{2}V[{x_0}]({x_*})+\frac {\widetilde{\delta}}{\mu}+ \frac {2\delta}{\mu}.
  \end{gather*}

\noindent
\dd{Similarly}, after the $2$nd restart ($N_2$ operations)
  \begin{gather*}
    V[\bar{x}_{N_{2}}]({x_*}) \leq \frac {1}{2} V[\bar{x}_{N_{1}}]({x_*})+\frac {\widetilde{\delta}}{\mu}+ \frac {2\delta}{\mu} \leq \frac {1}{4} V[{x_0}]({x_*})+\left(\frac {\widetilde{\delta}}{\mu}+ \frac {2\delta}{\mu}\right)\left(1+\frac{1}{2}\right).
  \end{gather*}

\noindent
After the $p$-th restart ($N_p$ operations)
  \begin{gather*}
    V[\bar{x}_{N_{p}}]({x_*}) \leq
    \frac{1}{2^p} V[{x_0}]({x_*})+ \left(\frac {\widetilde{\delta}}{\mu}+ \frac {2\delta}{\mu}\right)\left(1+\frac{1}{2} +...+\frac{1}{2^{p-1}}\right)<
  \end{gather*}
$$
< \frac{1}{2^p} V[{x_0}]({x_*}) + \frac {2\widetilde{\delta}}{\mu} + \frac {4\delta}{\mu}.
$$
Choose $p$ such that
  \begin{gather*}
    \frac {1}{2^p} V[{x_0}]({x_*}) \leq \varepsilon 
  \end{gather*}
After $p=\left\lceil \log_2 {\frac {R^2}{\e}} \right\rceil $ restarts we have
  \begin{gather*}
    V[\bar{x}_{N_{p}}]({x_*}) \leq \varepsilon + \frac {2\widetilde{\delta}}{\mu}+ \frac {4\delta}{\mu}.
  \end{gather*}

\noindent

The number of iterations $N_k \ (k=\overline{1,p})$ on the k-th restart of Algorithm \ref{Alg1} is estimated from \eqref{eqD}:
  \begin{equation}
     \frac {2L}{\mu N_k} \leq \frac{1}{2}, \ \ N_k \geq \frac {4L}{\mu}.
  \end{equation}
So, we can put $N_k = \left\lceil {\cfrac {4L}{\mu}} \right\rceil$ and \eqref{iterations} holds.
\end{proof}

We show that using the restart technique can also accelerate the work of non-adaptive version of Algorithm \ref{Alg2} ($L_{k+1} = L$) for ($\delta$, L)-model $\psi_{\delta}(x, y)$ w.r.t. norm $\|\cdot\|$ and relative $\mu$-strogly convex function $f$ in sense Definition \ref{defRelStronglyConvexRest}:
\begin{gather*}
  \mu V[x](y) + f(y)+ \psi_{\delta}(x, y) - \delta \leq f(x) \leq f(y) + \psi_{\delta}(x, y) + \frac{L}{2}\|x - y\|^2 + \delta.
\end{gather*}
for each $x, y \in Q$. By Theorem \ref{mainTheoremDLNew}:
\begin{equation}\label{FGeq1}
f(x_N) - f(x_*) \leqslant \dfrac{8LV[x_0](x_*)}{(N+1)^2}+\dfrac{8\widetilde{\delta}}{N+1} + 2N\delta.
\end{equation}
Consider the case of relatively $\mu$-strongly convex function $f$. We will use the restart technique to obtain the method for strongly convex functions.
By $\eqref{FGeq1}$ and Definition \ref{defRelStronglyConvexRest}:
\begin{equation}\label{FGeq2}
\mu V[x_{N_1}](x_*) \le f(x_{N_1}) - f(x_*) \le \dfrac{8LV[x_0](x_*)}{N^2} + \dfrac{8\widetilde{\delta}}{N} +2N\delta.
\end{equation}
Let's choose $N_1$ so that the following inequality holds:
\begin{equation}\label{FGeq4}
\dfrac{8\widetilde{\delta}}{N_1}+2N_1\delta \leq \dfrac{LV[x_0](x_*)}{N_1^2}.
\end{equation}
We restart method as
$$V[x_{N_1}](x_*)\leq\frac{V[x_0](x_*)}{4}.$$
From  \eqref{FGeq2}:
$$ \dfrac{9L}{\mu N_1^2} \le \dfrac{1}{4}, \quad N_1 \ge 6\sqrt{\frac{L}{\mu}}$$
Let's choose
\begin{equation}\label{FGeq3}
N_1 = \left \lceil 6\sqrt{\dfrac{L}{\mu}}\right \rceil.
\end{equation}
Then after $N_1$ iterations we restart method. Similarly, we restart after $N_2$ iterations, such that $V[x_{N_2}](x_*)\leq \frac{V[x_{N_1}](x_*)}{4}$. We obtain
$$N_2 = \left \lceil 6\sqrt{\frac{L}{\mu}}\right \rceil.$$
So, after $p$-th restart the total number of iterations:
$$M = p \cdot  \left \lceil 6\sqrt{\frac{L}{\mu}}\right \rceil.$$

Now let's consider how many iterations is needed to achieve accuracy $\varepsilon = f(x_{N_p}) - f(x_*)$. 
From \eqref{FGeq1} and \eqref{FGeq3} we take
$$p  = \left \lceil \log_4 \dfrac{\mu R^2}{\varepsilon} \right\rceil$$
and total number of iterations:
$$ M = \left \lceil \log_4 \dfrac{\mu R^2}{\varepsilon} \right\rceil \cdot \left \lceil 6\sqrt{\dfrac{L}{\mu}}\right \rceil.$$

Let's estimate the accuracy $\varepsilon$ we can achieve. For each $k = {1, p}$ we need to enforce the following inequality:
$$\frac{8\widetilde{\delta}}{N_k} + 2N_k\delta \leqslant \frac{LV[x_{k-1}](x_*)}{N_k^2},$$
where $N_k = \left \lceil 6\sqrt{\frac{L}{\mu}}\right \rceil$.
So, we can achieve the following accuracy:
$$\varepsilon \ge \dfrac{12\mu}{L}\left(9\delta\left \lceil \sqrt{\dfrac{L}{\mu}}\right \rceil^3 + \widetilde{\delta} \left \lceil \sqrt{\dfrac{L}{\mu}}\right \rceil \right).$$

\section{A proof of Theorem \ref{thmm1} for the case of inexactness for auxiliary problem}\label{UMPInexact}

For Algorithm \ref{Alg:UMPModel} we may also take into account inexactness for auxiliary problems on iterations (see Definition \ref{solNemirovskiy}).

\begin{algorithm}[ht]
\caption{Generalized Mirror Prox for VI}
\label{Alg:UMPModelInexact}
\begin{algorithmic}[1]
   \REQUIRE accuracy $\e > 0$, oracle error $\delta >0$,
   initial guess $L_{0} >0$, prox-setup: $d(x)$, $V[z] (x)$.
   \STATE Set $k=0$, $z_0 = \arg \min_{u \in Q} d(u)$.
   \FOR{$k=0,1,...$}
				\STATE Find the smallest $i_k \geq 0$ such that 
				\begin{equation}
                \begin{split}
                \hspace{-2em}\psi_{\delta}(z_{k+1}, z_{k})\leq \psi_{\delta}(z_{k+1}, w_{k})+\psi_{\delta}(w_k,z_k)+ L_{k+1}(V[z_k](w_k) + V[w_k](z_{k+1})) + \delta,
                \end{split}
                \end{equation}
			where 	$L_{k+1}=2^{i_k-1}L_{k}$ and 
			\begin{align}
			    w_k& =  {\arg\min_{x \in Q}}^{\widetilde{\delta}} \left\{\psi_{\delta}(x, z_k)+ L_{k+1}V[z_k](x) \right\}.
				\\
				z_{k+1}& =  {\arg\min_{x \in Q}}^{\widetilde{\delta}} \left\{\psi_{\delta}(x, w_k) + L_{k+1}V[z_k](x)\right\}.
			\end{align}
					\ENDFOR
		\ENSURE $\widehat{w}_N = \frac{1}{\sum_{k=0}^{N-1}\frac{1}{L_{k+1}}}\sum_{k=0}^{N-1}\frac{1}{L_{k+1}}w_k$.
\end{algorithmic}
\end{algorithm}

\begin{theorem}\label{thmm1inexact}
For Algorithm \ref{Alg:UMPModelInexact} the following inequalities hold
\begin{equation*}
\max\limits_{u\in Q}\left(-\frac{1}{S_N}\sum_{k=0}^{N-1}\frac{\psi_{\delta}(u,w_{k})}{L_{k+1}} \right)\leq \frac{2L \max_{u\in Q}V[z_0](u)}{N} + \delta + 2\widetilde{\delta},
\end{equation*}
\begin{equation*}
\max\limits_{u \in Q}\psi_{\delta}(\widehat{w}_N,u)
\leq \frac{2L \max_{u\in Q}V[z_0](u)}{N} + 2\delta + 2\widetilde{\delta}, 
\end{equation*}
\begin{equation*}
\text{} \; \widehat{w}_N:=\frac{1}{S_k}\sum_{i=0}^{k-1}\frac{w_{k}}{L_{k+1}}.
\end{equation*}
The method works no more than 
\begin{equation}
\left\lceil\frac{2L \max_{u\in Q}V[z_0](u)}{\varepsilon}\right\rceil
\end{equation}
iterations.
\end{theorem}

\begin{proof}
After $(k+1)$-th iteration ($k=0,1,2\ldots$) we have for each $u \in Q$:
$$\psi_{\delta}(w_k, z_k)\leqslant\psi(u, z_k) +L_{k+1}V[z_k](u)-L_{k+1}V[w_k](u)- L_{k+1}V[z_k](w_k) + \widetilde{\delta}$$
and
$$\psi_{\delta}(z_{k+1}, w_k)\leq\psi_{\delta}(u, w_k)+L_{k+1}V[z_k](u)-L_{k+1}V[z_{k+1}](u)-L_{k+1}V[z_k](z_{k+1}) + \widetilde{\delta}.$$
Taking into account \eqref{eqUMP23}, we obtain
$$
-\psi_{\delta}(u, w_k) \leq L_{k+1}V[z_k](u)-L_{k+1}V[z_{k+1}](u) + \delta + 2\widetilde{\delta}.
$$
So, the following inequality
$$-\sum_{k=0}^{N-1} \frac{\psi_{\delta}(u,w_{k})}{L_{k+1}} \leq V[z_0](u) - V[z_N](u) + S_N \cdot (\delta + 2\widetilde{\delta})$$
holds. By virtue of \eqref{eq20} and the choice of $L_{0}\leqslant 2L$, it is guaranteed that $$L_{k+1}\leqslant 2L\;\;\forall k=\overline{0,N-1}.$$
and we have
\begin{equation}
\max\limits_{u \in Q} \psi_{\delta}(\widehat{w}_N, u)\leqslant-\frac{1}{S_N}\sum_{k=0}^{N-1}\frac{\psi_{\delta}(u,w_{k})}{L_{k+1}} + \delta + 2\widetilde{\delta} \leq 
\end{equation}
$$
\leq \frac{2L \max_{u\in Q}V[z_0](u)}{N} + 2\delta + 2\widetilde{\delta}.
$$
\end{proof}

\section{On the concept of a \texorpdfstring{$(\delta, L)$}{dL}-model for saddle point problems}\label{SP}

The solution of variational inequalities reduces the so-called saddle problems, in which for a convex in $u$ and concave in $v$ functional $f(u,v):\mathbb{R}^{n_1+n_2}\rightarrow\mathbb{R}$ ($u\in Q_1\subset\mathbb{R}^{n_1}$ and $v\in Q_2\subset\mathbb{R}^{n_2}$) needs to be found such that:
\begin{equation}\label{eq31}
f(u_*,v)\leqslant f(u_*,v_*)\leqslant f(u,v_*)
\end{equation}
for arbitrary $u\in Q_1$ and $v\in Q_2$. Let $Q=Q_1\times Q_2\subset\mathbb{R}^{n_1+n_2}$. For $x=(u,v)\in Q$, we assume that $||x||=\sqrt{||u||_1^2+||v||_2^2}$ ($||\cdot||_1$ and $||\cdot||_2$ are the norms in the spaces $\mathbb{R}^{n_1}$ and $\mathbb{R}^{n_2}$). We agree to denote $x=(u_x,v_x),\;y=(u_y,v_y)\in Q$.

It is well known that for a sufficiently smooth function $f$ with respect to $u$ and $v$ the problem \eqref{eq31} reduces to VI with an operator
\begin{equation}\label{eq32}
g(x)=
\begin{pmatrix}
f_u'(u_x,v_x)\\
-f_v'(u_x,v_x)
\end{pmatrix}.
\end{equation}

For saddle-point problems we propose some adaptation of the concept of the ($\delta$, L)-model for abstract variational inequality (w.r.t. $V[y](x)$ or $\|\cdot\|$).

\begin{definition}
We say that the function $\psi_{\delta}(x,y)$ $(\psi_{\delta}:\mathbb{R}^{n_1+n_2}\times\mathbb{R}^{n_1\times n_2}\rightarrow\mathbb{R})$ is a $(\delta,L)$-model w.r.t. $V[y](x)$  for the saddle-point problem \eqref{eq31} if 
the following properties hold for each $x, y, z \in Q$:
\begin{enumerate}
\item[(i)] $\psi_{\delta} (x, y)$ convex in the first variable; 
\item[(ii)] $\psi_{\delta}(x,x)=0$;
\item[(iii)] ({\it abstract $\delta$-monotonicity}) 
\begin{equation}
\psi_{\delta}(x,y)+\psi_{\delta}(y,x)\leq \delta;
\end{equation}
\item[(iv)] ({\it generalized relative smoothness}) 
\begin{equation}
\psi_{\delta}(x,y)\leqslant\psi_{\delta}(x,z)+\psi_{\delta}(z,y)+ LV[z](x)+ LV[y](z)+\delta
\end{equation}
for some fixed values $L>0$, $\delta>0$;
\item[(v)] 
\begin{equation}\label{eq33}
f(u_y,v_x)-f(u_x,v_y)\leqslant-\psi_{\delta}(x,y) + \delta.
\end{equation}
\end{enumerate}
\end{definition}

\begin{example}
The proposed concept of the ($\delta$, L)-model for saddle-point problems is quite applicable, for example, for composite saddle problems of the form considered in the popular article \cite{chambolle2011first-order}:
\begin{equation}\label{eq34}
f(u,v)=\tilde{f}(u,v)+h(u)-\varphi(v)
\end{equation}
for some convex in $u$ and concave in $v$ subdifferentiable functions $\tilde{f}$, as well as convex functions $h$ and $\varphi$. In this case, you can put
\begin{equation}\label{eq35}
\psi_{\delta}(x,y)=\langle\tilde{g}(y),x-y\rangle+h(u_x)+\varphi(v_x)-h(u_y)-\varphi(v_y),
\end{equation}
where
$$
\tilde{g}(y)=
\begin{pmatrix}
\tilde{f}_u'(u_y,v_y)\\
-\tilde{f}_v'(u_y,v_y)
\end{pmatrix}.
$$

Indeed, from subgradient inequalities:
$$\tilde{f}(u_y,v_y)-\tilde{f}(u_x,v_y)\leqslant\langle-\tilde{f}_u'(u_y,v_y),u_x-u_y\rangle,$$
$$\tilde{f}(u_y,v_x)-\tilde{f}(u_y,v_y)\leqslant\langle \tilde{f}_v'(u_y,v_y),v_x-v_y\rangle.$$
Therefore, we have
$$\tilde{f}(u_y,v_x)-\tilde{f}(u_x,v_y)\leqslant-\langle\tilde{g}(y),x-y\rangle,$$
from where
$$f(u_y,v_x)-f(u_x,v_y)=\tilde{f}(u_y,v_x)+h(u_y)-\varphi(v_x)-\tilde{f}(u_x,v_y)-h(v_x)+\varphi(v_y)=$$
$$=\tilde{f}(u_y,v_x)-\tilde{f}(u_x,v_y)+h(u_y)+\varphi(v_y)-h(v_x)-\varphi(v_x)\leqslant$$
$$\leqslant-\langle\tilde{g}(y),x-y\rangle+h(u_y)+\varphi(v_y)-h(v_x)-\varphi(v_x)=-\psi_{\delta}(x,y).$$
\end{example}

Theorem \ref{thmm1inexact} implies
\begin{theorem}
If for the saddle problem \eqref{eq31} there is a $(\delta,L)$-model $\psi(x,y)$ w.r.t. $V[y](x)$, then after stopping the algorithm we get a point
\begin{equation}\label{eq36}
\widehat{y}_N=(u_{\widehat{y}_N},v_{\widehat{y}_N}):=(\widehat{u}_N, \widehat{v}_N):=\frac{1}{S_N}\sum_{k=0}^{N-1}\frac{y_{k}}{L_{k+1}},
\end{equation}
for which the inequality is true:
\begin{equation}\label{eq37}
\max_{v\in Q_2}f(\widehat{u}_N, v)-\min_{u\in Q_1}f(u, \widehat{v}_N)\leqslant \frac{2L \max_{(u, v) \in Q} V[u_0, v_0](u, v)}{N} +2\tilde{\delta}+\delta.
\end{equation}
\end{theorem}

\section{Modelling for Strongly Monotone VI}
\label{UMPStrongApp}

We also can consider $\mu$-strongly monotone ($\delta$, L)-model for VI with the following more strong version of \eqref{eq:abstr_monot} :
\begin{equation}\label{eq:abstr_trong_monot}
\psi_{\delta}(x,y) + \psi_{\delta}(y,x) + \mu\|y - x\|^2\leqslant0\;\;\forall x,y\in Q
\end{equation}
for some fixed number $\mu>0$ (here we put $\delta = 0$). Also we assume that $\psi_{\delta}(x, y)$ is continuous by $x$ and $y$. We slightly modify the assumptions on prox-function $d(x)$. Namely, we assume that $0 = \arg \min_{x \in Q} d(x)$ and that $d$ is bounded on the unit ball in the chosen norm $\|\cdot\|$, that is
\begin{equation}
d(x) \leq \frac{\Omega}{2}, \quad \forall x\in Q : \|x \| \leq 1,
\label{eq:dUpBound}
\end{equation}
where $\Omega$ is some known constant. Note that for standard proximal setups, $\Omega = O(\ln \text{dim}E)$. Finally, we assume that we are given a starting point $x_0 \in Q$ and a number $R_0 >0$ such that $\| x_0 - x_* \|^2 \leq R_0^2$, where $x^*$ is the solution to abstract VI. The procedure of restating of Algorithm \ref{Alg:UMPModel} restating is applicable for abstract strongly monotone variational inequalities.

\begin{algorithm}
\caption{Restarted Generalized Mirror Prox}
\label{Alg:RUMP}
\begin{algorithmic}[1]
   \REQUIRE accuracy $\e > 0$, $\mu >0$, $\Omega$ s.t. $d(x) \leq \frac{\Omega}{2} \ \forall x\in Q: \|x\| \leq 1$; $x_0, R_0 \ s.t. \|x_0-x_*\|^2 \leq R_0^2.$
      \STATE Set $p=0,d_0(x)=d\left(\frac{x-x_0}{R_0}\right)$.
   \REPEAT
			\STATE Set $x_{p+1}$ as the output of Algorithm \ref{Alg:UMPModel} after $N_p$ iterations for monotone case with accuracy $\mu\e/2$, prox-function $d_{p}(\cdot)$ and stopping criterion $\sum_{k=0}^{N_p-1}\frac{1}{L_{k+1}} \geq \frac{\Omega}{\mu}$.
			\STATE Set $R_{p+1}^2 = R_0^2 \cdot 2^{-(p+1)} + \frac{(1-2^{-p})\e}{2}$.
			\STATE Set $d_{p+1}(x) \leftarrow d\left(\frac{x-x_{p+1}}{R_{p+1}}\right)$.
			\STATE Set $p=p+1$.
			\UNTIL			
			$p > \log_2\frac{2R_0^2}{\e}$	
		\ENSURE $x_{p+1}$.
\end{algorithmic}
\end{algorithm}

\begin{theorem}
\label{Th:RUMPCompl}
    Assume that $\psi$ is satisfied to \eqref{eq:abstr_trong_monot}. Also assume that the prox function $d(x)$ satisfies \eqref{eq:dUpBound} and the starting point $x_0 \in Q$ and a number $R_0 >0$ are such that $\| x_0 - x_* \|^2 \leq R_0^2$, where $x_*$ is the solution to \eqref{eq115}. Then, for $p\geq 0$
		\[
		\|x_p - x_*\|^2 \leq R_0^2\cdot 2^{-p} + \frac{\e}{2} 
		\]
		and the point $x_p$ returned by natural analogue of Algorithm \ref{Alg:RUMP} with restarts of Algorithm \ref{Alg:UMPModel} satisfies $\|x_p-x_*\|^2\leq \e$. The total number of iterations of the inner Algorithm \ref{Alg:UMPModel} does not exceed
    \begin{equation}\label{eq_abst_strong_monot}
         \left\lceil \frac{2L\Omega}{\mu}\cdot \log_2 \frac{2 R_0^2}{\e}\right\rceil,
    \end{equation}
where $\Omega$ is satisfied to \eqref{eq:dUpBound}.
\end{theorem}
\begin{proof}
We show by induction that, for $p \geq 0$,
\[
\|x_p - x_*\|^2 \leq R_0^2\cdot 2^{-p} + \frac{(1-2^{-p})\e}{2},
\]
which leads to the statement of the Theorem.
For $p=0$ this inequality holds by the Theorem assumption. Assuming that it holds for some $p\geq 0$, our goal is to prove it for $p+1$ considering the outer iteration $p+1$.
Observe that the function $d_{p}(x)$ defined in Algorithm \ref{Alg:RUMP} is 1-strongly convex w.r.t. the norm $\|\cdot\| / R_{p}$.

This means that, at each step $k$ of inner Algorithm \ref{Alg:UMPModel}, $L_{N_p}$ changes to $L_{N_p} \cdot R_{p}^2$. Using the definition of $d_{p}(\cdot)$ and \eqref{eq:dUpBound}, we have, since $x_p = \arg \min_{x \in Q} d_p(x)$
\[
	V_{p}[x_{p}](x_*) = d_{p}(x_{*}) - d_{p}(x_{p}) - \la \nabla d_{p}(x_{p}), x_{*} - x_{p} \ra \leq  d_{p}(x_{*}) \leq \frac{\Omega}{2}.
\]
Denote by $$S_{N_p}:= \sum_{k=0}^{N_p-1}\frac{1}{L_{k+1}}.$$

Thus, by Theorem \ref{thmm1}, taking $u = x_*$, we obtain 
\[
- \frac{1}{S_{N_p}} \sum_{k=0}^{N_p-1} \frac{\psi_{\delta}(x_*, w_k)}{L_{k+1}} \leq  \frac{R_{p}^2V_{p}[x_{p}](x_{*})}{S_{N_p}}   + \frac{\mu\e}{4} \leq \frac{\Omega R_p^2}{2S_{N_{p}}} + \frac{\mu\e}{4}.
\]
Since the operator $\psi_{\delta}$ is continuous and abstract monotone, we can assume that the solution to weak VI \eqref{eq13} is also a strong solution and
\[
- \psi_{\delta}(w_k, x_*) \leq 0, \quad k=0,...,N_p-1.
\]
This and \eqref{eq:abstr_trong_monot} gives, that for each $k=0,...,N_p-1$,
\[
- \psi_{\delta} (x_*, w_k) \geq - \psi_{\delta}(x_*, w_k) - \psi_{\delta}(w_k, x_*) \geq \mu\|w_k-x_*\|^2.
\]
Thus, by convexity of the squared norm, we obtain
\begin{align}
\mu \|x_{p+1}-x_*\|^2 &= \mu \left\|\frac{1}{S_{N_p}} \sum_{k=0}^{N_p-1}  \frac{w_k}{L_{k+1}}-x_*\right\|^2 \leq \frac{\mu}{S_{N_p}} \frac{1}{L_{k+1}} \sum_{k=0}^{N_p-1} \|w_k-x_*\|^2 \notag \\
&  \leq - \frac{1}{S_{N_p}} \sum_{k=0}^{N_p-1} \frac{\psi_{\delta}(x_*, w_k)}{L_{k+1}} \leq \frac{\Omega R_p^2}{2 S_{N_p}} + \frac{\mu \varepsilon}{4}. \notag
\end{align}
Using the stopping criterion $S_{N_p} \geq \frac{\Omega}{\mu}$, we obtain
$$
\|x_{p+1}-x_*\|^2 \leq \frac{R_p^2}{2} + \frac{\varepsilon}{4} = \frac{1}{2} \left(R_0^2 \cdot 2^{-p} + \frac{(1-2^{-p})\e}{2} \right) + \frac{\varepsilon}{4} = 
$$
$$
= R_0^2 \cdot 2^{-(p+1)} + \frac{(1-2^{-p})\e}{2}, 
$$
which finishes the induction proof.
\end{proof}

\end{document}